\documentclass[10pt,reqno]{amsart}
\usepackage{booktabs}
\usepackage{amsmath,amssymb,amsthm,amsfonts}
\usepackage{cite}
\usepackage[height=190mm,width=130mm]{geometry}
\usepackage[matrix,arrow,curve]{xy}
\usepackage{amscd}
\usepackage{graphicx, tikz}
\usepackage{comment}
\usepackage{etoolbox}
\usepackage{mathtools}
\usepackage[hidelinks]{hyperref}
\usepackage{color}
\usepackage{cancel}

\definecolor{brown}{RGB}{150,100,0}

\definecolor{purple}{RGB}{150,0,100}
\definecolor{grey}{RGB}{128,128,128}

\mathchardef\mhyp="2D

\def\NABLA#1{{\mathop{\nabla\kern-.5ex\lower1ex\hbox{$#1$}}}}
\def\Nabla#1{\nabla\kern-.5ex{}_#1}

\theoremstyle{plain}
\newtheorem{theorem}{Theorem}[section]
\newtheorem{lemma}[theorem]{Lemma}
\newtheorem{corollary}[theorem]{Corollary}
\newtheorem{proposition}[theorem]{Proposition}

\theoremstyle{definition}
\newtheorem{definition}[theorem]{Definition}

\newtheorem{remark}[theorem]{Remark}

\newcommand{\ps}{{\raise 1pt\hbox{\tiny (}}}

\newcommand{\pss}{{\raise 1pt\hbox{\tiny [}}}
\newcommand{\pdd}{{\raise 1pt\hbox{\tiny ]}}}
\newcommand{\pd}{{\raise 1pt\hbox{\tiny )}}}

\newcommand{\bs}{{\raise 1pt\hbox{\tiny [}}}
\newcommand{\bd}{{\raise 1pt\hbox{\tiny ]}}}

\def\cross{\mathinner{\mathrel{\raise0.8pt\hbox{$\scriptstyle>$}}
		\joinrel\mathrel\triangleleft}}

\def\K{\mathcal{K}}

\newcommand{\wt}{\mbox{\rm wt}\ }

\newcommand{\nc}{\newcommand}
\nc{\cali}{\mathcal}
\nc{\on}{\operatorname}
\nc{\Wick}{{\mathbf :}}
\nc{\ddz}{\frac{\partial}{\partial z}}
\nc{\ch}{\mbox{ch}}
\nc{\Oo}{{\cali O}}
\nc{\cond}{|\,}
\nc{\bib}{\bibitem}
\nc{\pone}{\mathbb{P}^1}
\nc{\pa}{\partial}
\nc{\arr}{\rightarrow}
\nc{\larr}{\longrightarrow}
\nc{\ket}{\rangle}
\nc{\bra}{\langle}
\nc{\gam}{\bar{\gamma}}
\nc{\q}{\widetilde{Q}}
\nc{\ep}{\epsilon}
\nc{\su}{\widehat{{\mathfrak{s}}{\mathfrak{l}}}_2}
\nc{\sw}{{\mathfrak{s}}{\mathfrak{l}}}
\nc{\h}{\mathfrak{h}}
\nc{\n}{\mathfrak{n}}
\nc{\ab}{\mathfrak{a}}
\nc{\is}{\mathbf{i}}
\nc{\js}{\mathbf{j}}
\nc{\bi}{\bibitem}
\nc{\He}{{\cali H}}
\nc{\inv}{^{-1}}
\nc{\ol}{\overline}
\nc{\wh}{\widehat}
\nc{\dst}{\displaystyle}
\nc{\delt}{\partial_t}
\nc{\ddt}{\frac{\partial}{\partial t}}
\nc{\delx}{\partial_x}
\nc{\mb}{\mathbf}
\nc{\mf}{\mathfrak}
\nc{\mbb}{\mathbb}
\nc{\Ctt}{\mathbb{C}((t))}
\nc{\Ct}{\mathbb{C}[t,t\inv]}
\nc{\ghat}{\wh{\Gamma}}

\nc{\un}{\underline}
\nc{\mc}{\mathcal}
\nc{\BB}{{\mc B}}
\nc{\bb}{{\mf b}}
\nc{\kk}{{\mf k}}
\nc{\frob}{\times}
\nc{\sm}{\setminus}
\nc{\Pp}{\mathbb{P}^1}
\nc{\Aa}{\mc A}
\nc{\AutO}{\on{Aut}\Oo}
\nc{\AUTO}{\un{\on{Aut}}\Oo}
\nc{\AUTK}{\un{\on{Aut}}\K}
\nc{\Heout}{\He_{\out}}
\nc{\Hetil}{{\widetilde\He}}
\nc{\wb}{\overline}
\nc{\Res}{\on{Res}}
\nc{\pitil}{\Pi}
\nc{\Ctil}{\wt{C}}
\nc{\auto}{\on{Aut} \Oo}
\nc{\phitil}{\wt{\phi}}
\nc{\gz}{\Gamma_{\vec z}}
\nc{\tensorM}{\bigotimes_{i=1}^N{\mathbb M}_i}
\nc{\tensorW}{\bigotimes_{i=1}^N W_{\nu_i,k}}
\nc{\out}{\on{out}}
\nc{\m}{\mathfrak{m}}
\nc{\gx}{\Gamma^0_{\vec x}}
\nc{\hx}{\He^0_{\vec x}}
\nc{\tensorpi}{\pi_{\nu_1,\ldots,\nu_N}^\kappa}
\nc{\Phizw}{\Phi_{\vec w}({\vec z})}
\nc{\Pro}{\mathbb{P}}
\nc{\De}{\Delta}
\nc{\us}{\underset}
\nc{\Ll}{\mc L}
\nc{\dR}{\on{dR}}
\nc{\T}{\mc T}
\nc{\Xn}{\overset{\circ}X{}^n}
\nc{\Dn}{\overset{\circ}D{}^n}
\nc{\Dxn}{\overset{\circ}D{}^n_x}
\nc{\varphitil}{\wt{\varphi}}
\nc{\lf}{\mathfrak{l}}
\nc{\GL}{{}^L G}
\nc{\Vir}{\on{Vir}}

\numberwithin{equation}{section}

\begin{document}

\title[Azumaya descent and \'etale-local cyclotomic comparison]
{Azumaya descent and \'etale-local cyclotomic comparison: a noncommutative filtration}

\author{H\^ong V\^an L\^e$^\dagger$}
\address{Institute of Mathematics, Czech Academy of Sciences, \v{Z}itn\'{a} 25, Prague, Czech
		Republic}
\email{}

\author{Alexander Zuevsky}
\address{Institute of Mathematics, Czech Academy of Sciences, \v{Z}itn\'{a} 25, Prague, Czech
		Republic}
\email{zuevsky@yahoo.com}
\keywords{Topological Hochschild homology ($THH$); Azumaya algebras; Derived Azumaya algebras; Cyclotomic spectra; Bhatt-Morrow-Scholze (BMS) filtration; \'{E}tale descent; Noncommutative motives; $\mathbb{E}_\infty$-rings, Prismatic cohomology.}
\subjclass[2020]{Primary:16E40, 19D55, 55P43, Secondary:14F30, 16H05, 18E30}

\begin{abstract}
Bhatt-Morrow-Scholze motivic filtration on topological Hochschild homology ($THH$) 
is extended to Azumaya algebras over quasisyntomic rings. We prove the Azumaya descent theorem:
 $THH$ of such an algebra is equivalent, as an underlying spectrum, to $THH$ of its center,
 normalized on $\pi_0$ by the reduced trace  instead of   the unit map. This comparison 
$\Phi_A$ is not shown to be \'etale-locally cyclotomic. 
  But what is cyclotomic 
\'etale-locally is a separate Morita trace, agreeing with $\Phi_A$ only on $\pi_0$. 
The resulting transported filtration is exactly the center  BMS filtration, well 
defined up to the automorphism ambiguity of $\Phi_A$. It is insensitive to the Brauer class, 
reflected instead by $K$-theoretic and module-categorical pairings $THH$ does not see. 
We establish the analogous statement for derived Azumaya algebras over 
connective $\mathbb{E}_\infty$-rings, under a connectivity hypothesis. 
\end{abstract}

\maketitle

\footnotetext[1]{$^\dagger$ We deeply mourn the sudden passing of our co-author
 and dear friend, H\^ong V\^an L\^e,
 who suddenly passed away during the completion of this work. 
 We are profoundly grateful for her wisdom, leadership, 
and invaluable contributions to this research.}
	\tableofcontents

\section{Introduction}

The recent synthesis of $p$-adic Hodge theory by Bhatt, Morrow, and Scholze has profoundly advanced the understanding of cohomology theories: Bhatt, Morrow, and Scholze constructed a motivic filtration on topological Hochschild homology ($THH$) and its cyclotomic variants \cite{bms, bhm, nik}.
  The prismatic cohomology later introduced by Bhatt and Scholze \cite{bs} gives the modern, general setup for the crystalline objects computing its graded pieces.
  Extending this powerful development to noncommutative spaces, such as deformation-quantized varieties, derived categories of coherent sheaves, or matrix factorizations, is a central objective driven by the paradigm of noncommutative motives \cite{tabuada, bgt}.

While the commutative theory provides a unified setting connecting crystalline, de Rham, and $p$-adic \'{e}tale cohomology theories, passing to noncommutative deformations presents significant structural challenges \cite{dm}.
  Specifically, formulating a noncommutative analogue of the Bhatt-Morrow-Scholze (BMS) filtration requires understanding the filtered deformation of the underlying cyclotomic structure \cite{nik}.
  This involves a highly non-abelian deformation problem, complete with equivariant Frobenius maps and coherent restrictions, that standard Maurer-Cartan formalisms cannot fully capture \cite{Kaledin2016}.

From a physical perspective, the necessity of a noncommutative prismatic cohomology is motivated by the behavior of quantum observable algebras under arithmetic reduction modulo $p$ \cite{bk}. Formal deformation quantization in characteristic $p$ exhibits radically different features than in characteristic zero.
  For example, modular Weyl algebras cease to be simple and acquire a large center, effectively becoming Azumaya algebras over a Frobenius twist \cite{revoy, bmr}. This large-center arithmetic  phenomenon physically parallels rational magnetic flux setups, and
the integer and fractional quantum Hall effects \cite{TKNN, kmmzz}.

In this paper, we bypass the conjectural general noncommutative prismatic theory by
establishing a noncommutative filtration for a broad and natural class of algebras,
namely Azumaya algebras over quasisyntomic commutative rings. We state the scope of what
follows, since the phrase ``cyclotomic descent'' is easy to overread: we prove
that for these algebras, $THH$ descends to the center as an underlying spectrum, and
 separately, \'etale-locally, a Morita trace comparison is a cyclotomic
equivalence. These two comparisons are not claimed to agree beyond $\pi_0$. 
 Only the
former (the underlying equivalence, together with a chosen witness $\Phi_A$
 non-canonical data) is what lets the commutative BMS filtration be transported to
the noncommutative setting (Definition \ref{deftransported}). A global cyclotomic
statement holds only when $A$ admits a global splitting module.

Furthermore, we transition into higher algebra to generalize the discrete case \cite{toen, ag}.  We establish an analogous descent theorem for derived Azumaya algebras over connective
$\mathbb{E}_\infty$-rings within the setup of stable $\infty$-categories \cite{bgt}.
By demonstrating that $THH$ descends for derived Azumaya algebras (again, as an underlying
spectrum, and separately, \'etale-locally, via a cyclotomic Morita trace) we provide the
underlying comparison that a transport of the BMS filtration to this homotopical setting
would require. Unlike in the discrete case, we do not carry out that transport  here
(see Table \ref{tabsummary}).

We separate what is classical here from what is new. That $THH$, as a
localizing invariant, is Morita invariant, and that Azumaya algebras are \'etale-locally
matrix algebras, are known. The former is due to Blumberg-Gepner-Tabuada \cite{bgt} and
underlies our Lemma \ref{lemmorita-pic}(i). The latter, in the derived setting, is
Antieau-Gepner local triviality theorem \cite{ag}. Combined naively, these give only
that $THH(A)$ is locally equivalent to $THH(Z)$ on a splitting cover, which is
already classical in spirit (compare Corti\~nas-Weibel identification of ordinary
Hochschild homology \cite{CW}). What is new in Theorem \ref{thmazumaya} is: (i) a
 global equivalence of underlying spectra, valid whether or not $A$ splits, via
the invertible-object argument of Section \ref{azumayo}  instead of   local gluing; (ii)
its normalization by the reduced trace  instead of   the unit map, which we show differ
by multiplication by $\deg A$ (Remark \ref{remunit}); (iii) the resulting transported
BMS filtration, with graded pieces identified explicitly; and (iv) the 
delimitation,  instead of   blurring, of how far the \'etale-local cyclotomic statement is
from a global one (Remark \ref{remcoherence}). Theorem \ref{thmderived} is the further,
  new extension of (i)-(iv) to the derived setting, which requires the
connectivity hypothesis and Lemma \ref{lempi0azumaya} that have no classical analogue.

\subsection{Structure of the paper}
	This paper is organized as follows. Section \ref{ncbms} introduces the foundational setup for the noncommutative Bhatt-Morrow-Scholze (BMS) filtration and Moyal deformations, establishing the structural imperatives from both mathematical and physical perspectives.

 In Section \ref{azumayo}, we develop the theory of Azumaya descent for topological Hochschild homology, providing a filtration for Azumaya algebras over quasisyntomic rings.

Section \ref{doloto} transitions into higher algebra. It generalizes the descent theorem to derived Azumaya algebras over connective $\mathbb{E}_\infty$-rings within the setup of stable $\infty$-categories.

The underlying equivalence of $THH$-spectra in Section \ref{doloto} is assembled from
established $\infty$-categorical ingredients (the Morita invariance of $THH$ as a
localizing invariant, \'etale descent, and the triviality of the Picard action, due to
To\"en and Antieau-Gepner, generalizing the classical theorem of Corti\~nas and Weibel). 
  The contribution of that section is the derived generalization: the trace
 argument, the connectivity hypothesis it requires, and the delimitation of the
\'etale-local cyclotomic statement.

\section{The noncommutative BMS filtration and Moyal deformations}
	\label{ncbms}

	Let $k$ be a perfect field of characteristic $p>2$, let $A$ be a smooth commutative
	$k$-algebra with a constant Poisson structure (e.g., $A=k[x_1,\ldots,x_{2n}]$ with its
	symplectic bracket). Let $\mathcal{A}_\star$ denote a noncommutative deformation of
	$A$ by a formal parameter $\hbar$, arising by reduction modulo $p$ from a quantization
	defined over $W(k)[[\hbar]]$. This avoids the factorial denominators obstructing the
	naive Moyal formula in characteristic $p$. At first order, 
	$a\star b=ab+\frac{\hbar}{2}\{a,b\}+O(\hbar^2)$. We refer to \cite{Kaledin2016} for
	Kaledin noncommutative Hochschild-Witt complex,  relevant to extensions of
	the present setup.

	For commutative quasisyntomic $k$-algebras, Bhatt-Morrow-Scholze construct motivic
	filtrations on $THH$, $TC^-$, $TP$, and $TC$ via the following theorem due to Bhatt-Morrow-Scholze \cite{bms}. 

	\begin{theorem}
		\label{thmhm}
		Let $R$ be a quasisyntomic ring (for instance a smooth algebra over a perfect field $k$ of
		characteristic $p$). There are complete, exhaustive, multiplicative $\mathbb{Z}$-indexed
		filtrations $F^\bullet_{\mathrm{BMS}}$ on the $p$-completions of $THH(R)$, $TC^-(R)$, $TP(R)$,
		and $TC(R)$, functorial in $R$ and constructed compatibly with the cyclotomic
		structure maps $\varphi$ and $\mathrm{can}$,  with
		\[
			\mathrm{gr}^nTHH(R;\mathbb{Z}_p)\simeq\mathrm{gr}^n_N\widehat\Delta_R\{n\}[2n],
			\qquad
			\mathrm{gr}^nTC^-(R;\mathbb{Z}_p)\simeq\mathcal N^{\ge n}\widehat\Delta_R\{n\}[2n],
		\]
		\[
			\mathrm{gr}^nTP(R;\mathbb{Z}_p)\simeq\widehat\Delta_R\{n\}[2n],
		\]
		where $\widehat\Delta_R\{n\}$ is the ($n$-th Breuil-Kisin twisted) Nygaard-completed
		prismatic complex of $R$, $\mathcal N^{\ge n}$ its Nygaard filtration, and $\mathrm{gr}^n_N$
		the associated graded. 
\end{theorem}

Not that this motivic filtration is not  a
		filtration by cyclotomic subspectra of $THH$, see \cite{an, ar} for what a  
		cyclotomic refinement requires, and cf. Section \ref{physo}. 
The graded pieces of $TC$ are correspondingly given by syntomic
		cohomology. In characteristic $p$, $\widehat\Delta_R\{n\}$ is identified with Nygaard-completed
		crystalline cohomology, and when $R$ is smooth over $k$ its Nygaard-graded pieces are computed
		by the de Rham-Witt complex of $R$ 
\cite{illusie, illraynaud, ekedahl, cartier, di}.

	Since $THH(R;\mathbb{Z}_p)$ is connective, $F^n_{\mathrm{BMS}}THH(R;\mathbb{Z}_p)
	\simeq THH(R;\mathbb{Z}_p)$ for $n\le 0$, thus the $\mathbb{Z}$-indexed filtration on
	$THH$ is exhaustive already from $n=0$. This is the convention used for
	$THH(A;\mathbb{Z}_p)$ in Definition \ref{deftransported} below. The filtrations on
	$TC^-(R;\mathbb{Z}_p)$ and $TP(R;\mathbb{Z}_p)$ use negative indices, as neither
	invariant is connective. Note that none of these graded pieces are, in
	general, 
Eilenberg-MacLane spectra on the individual de Rham-Witt modules $W\Omega_A^i$. 

	Extending this filtration to noncommutative deformations $\mathcal{A}_\star$ is the
	organizing problem of this paper. The obstacle is that a cyclotomic structure
	requires, for every $n\ge 1$, a cyclotomic Frobenius $\varphi_p:THH(\mathcal{A}_\star)\to
	THH(\mathcal{A}_\star)^{tC_p}$ together with the $C_{p^n}$-equivariant restriction maps
	$R:THH(\mathcal{A}_\star)^{C_{p^n}}\to THH(\mathcal{A}_\star)^{C_{p^{n-1}}}$ that $\varphi_p$
	induces on fixed points. In the sense of Nikolaus-Scholze, $R$ is formally
	recovered from $\varphi_p$ for bounded-below spectra \cite{nik}. 
         But we still name the 
	two separately below since they have different sources and targets. 
These are two distinct
	pieces of structure, and we use
	``Frobenius'' only for the former below, satisfying
	coherences internal to equivariant stable homotopy theory that an ordinary ($A_\infty$
	or $L_\infty$) deformation-theoretic formalism for the star product does not capture,  
	a separate and harder problem than deforming $\mathcal{A}_\star$ at the $E_1$-level
	alone, which we do not resolve here. We get the resulting picture as a spectral
	sequence of schematic form
	\begin{equation}
		E_2^{i,j}=\mathcal{NC}\text{-}\mathrm{Prismatic}^{i,j}(\mathcal{A}_\star)
		\implies\pi_{i+j}THH(\mathcal{A}_\star;\mathbb{Z}_p).
	\end{equation}
	Here $\mathcal{NC}\text{-}\mathrm{Prismatic}^{i,j}$ an unconstructed noncommutative
	refinement of Bhatt-Morrow-Scholze crystalline reduction, targeting $THH$  instead of  
	$TR$ or $TC$. A spectral sequence converging to $TR_*$ or $TC_*$ would additionally
	require filtering the full cyclotomic structure, 
which is the same open problem. We do not 
	give this construction here. Sections \ref{azumayo}-\ref{doloto} instead develop, on
	the commutative side, the compatibilities any such theory would have to specialize to.

	Two motivations, one mathematical and one physical, underlie this picture. Mathematically,
	Hochschild and cyclic homology are the classical noncommutative extensions of the de Rham
	complex \cite{tabuada, bgt}.
  $\mathcal{NC}\text{-}\mathrm{Prismatic}$ is a candidate for 
	noncommutative-motivic refinement bridging this with Hesselholt-Madsen de Rham-Witt
	computations \cite{hmwitt}. Physically, the Moyal algebra $\mathcal{A}_\star$ models a
	two-dimensional electron gas in a magnetic field \cite{TKNN, kmmzz}. Reducing an integral
	Weyl algebra modulo $p$ (instead of   reducing the flux, which would first require
	choosing an arithmetic model for it) produces a modular Weyl algebra whose $p$-th powers
	become central. 
Namely, retaining the deformation while reducing modulo $p$ means
	base-changing only the coefficients, $\mathcal{A}_\star=A_\hbar\otimes_{W(k)}k$, a
	$k[[\hbar]]$-algebra. 
Tensoring instead over $W(k)[[\hbar]]$ would set $\hbar\mapsto0$
	and collapse the star product back to the commutative one. At the Frobenius-constant
	specialization relevant to the examples below (Remark \ref{remfrobconst}), the center of
	$\mathcal{A}_\star$ contains the subring generated by the $p$-th powers of all elements.
	On the locus, and under the finiteness hypotheses, 
needed for this containment to be an
	equality realizing $\mathcal{A}_\star$ as a finite module over that subring 
(verified for the Weyl algebra example below)  
	 this exhibits $\mathcal{A}_\star$ as an Azumaya algebra over the
	Frobenius twist $X^{(1)}$,  
the structure underlying crystalline differential operators in 
	characteristic $p$ \cite{bmr} and Kaledin noncommutative Cartier isomorphism
	\cite{Kaledin2016}.
 The quantum torus example of Section \ref{ssectorus} is Azumaya by a different, prime-to-$p$ mechanism. The center of the isomorphism
  is generated by $m$-th powers with
	$p\nmid m$, not by $p$-th powers, and it is not a Frobenius twist.  
	 
\subsection{Structural requirements for the noncommutative filtration}
	\label{physo}

	We mention here, for later reference, the structural properties any construction of the
	filtration proposed above would have to satisfy, together with the physical invariant
	such a filtration is meant to control.

	\smallskip
	\noindent{\rm(A)} \emph{Existence and compatibility.} For $\mathcal{A}_\star$ in the class of
	noncommutative deformations under consideration, there should exist a complete, exhaustive,
	multiplicative $\mathbb{Z}$-indexed filtration $F^\bullet$ on $THH(\mathcal{A}_\star;\mathbb{Z}_p)$
	 (and compatibly on $TC^-$, $TP$, $TC$) restricting to the Bhatt-Morrow-Scholze filtration of
	Theorem \ref{thmhm} whenever $\mathcal{A}_\star$ is commutative.

	\smallskip
	\noindent{\rm(B)} \emph{Morita invariance.} This filtration should depend only on
	$\mathrm{Perf}(\mathcal{A}_\star)$, i.e., be invariant under Morita equivalence of
	$\mathcal{A}_\star$, so that it descends to an invariant of the underlying noncommutative space.

	\smallskip
	\noindent{\rm(C)} \emph{Equivariant refinement.} The graded pieces
	$\mathrm{gr}^iF^\bullet THH(\mathcal{A}_\star;\mathbb{Z}_p)$ of {\rm(A)} lift, for every
	$n$, to $C_{p^n}$-equivariant spectra, compatibly with restriction to $C_{p^{n-1}}$
	and with the Morita equivalence of {\rm(B)}, so that the resulting fixed-point data
	assembles into a Mackey functor on the orbit category of $\mathbb{Z}/p^r$. This is a
	 separate requirement from {\rm(A)}-{\rm(B)}: neither supplies equivariant
	refinements, restriction/transfer maps, or an equivariant pairing with $K_0$.

	\smallskip
	For a gapped physical system modelled on such an algebra, for instance a two-dimensional
	electron gas of integer or fractional quantum Hall type \cite{TKNN, kmmzz}, the relevant
	invariant is built from the Fermi projection. The spectral projection
	$p_F\in M_\infty(\mathcal{A}_\star)$ onto states below the Fermi level, with class
	$[p_F]\in K_0(\mathcal{A}_\star)$.
 Given {\rm(A)}-{\rm(C)}, the equivariant refinements
	of {\rm(C)}, together with the pairing of $[p_F]$ with the cyclic and syntomic classes
	of {\rm(A)}, assemble into a Mackey functor which we denote
	$\mathcal{I}^i_r(\mathcal{A}_\star)$. 
 Its restriction and transfer maps are built from,
	but should not be identified with, the cyclotomic restriction $R$ and the
	cyclotomic Frobenius $\varphi_p$.  
 Note that {\rm(C)} asks only that they agree after passing to the
	associated graded, where the motivic filtration is in general not a filtration by
	cyclotomic subspectra \cite{an, ar}. For the closely related equivariant motivic filtration of
	$THH$ of polynomial algebras, from which such Mackey functors can be extracted,
 	see \cite{HKS, hkszou}; for a Mackey-functor refinement of the commutative BMS filtration, see \cite{HVLZ}.

		Suppose {\rm(A)}-{\rm(C)} hold for $\mathcal{A}_\star$, and let
		$(\mathcal{A}_\star,[p_F])$ and $(\mathcal{A}_\star',[p_F'])$ be two marked gapped systems
		related by a Morita equivalence under which $[p_F]$ 
transports to $[p_F']$. Note that the class of a Fermi projection is not Morita-invariant data. 
It must be specified which class on one side the equivalence carries to which class on the
		other. 

	\begin{proposition}
		\label{propinvariance}
The specified Morita equivalence carries the pairing of $[p_F]$
		 with the transported cyclic and syntomic classes to the pairing of $[p_F']$, and
		identifies the Mackey functors $\mathcal{I}^i_r(\mathcal{A}_\star)\simeq
		\mathcal{I}^i_r(\mathcal{A}_\star')$.
	\end{proposition}

	\begin{proof}
		{\rm(A)} supplies the filtration and its graded pieces, the classes paired  with  a
		Fermi projection. {\rm(C)} supplies the $C_{p^n}$-equivariant structure from which
		$\mathcal{I}^i_r$ is assembled and its compatibility with the Morita equivalence of
		{\rm(B)}. Since $[p_F']$ was taken to be the transport of $[p_F]$, the identification
		of {\rm(B)}-{\rm(C)} carries the pairing and the Mackey functor for $[p_F]$ to those
		for $[p_F']$.
	\end{proof}

	We do not verify {\rm(A)}-{\rm(C)} in general in this note. 
Supplying them is exactly the
	content that a construction of $\mathcal{NC}\text{-}\mathrm{Prismatic}^{i,j}$ would have to
	deliver. 
On can compare \cite{Kadeishvili1988} for the classical $A_\infty$-algebra structure on Hochschild cohomology, and \cite{Le2026} for related cyclic/$C_\infty$-structures.  Theorem
	\ref{thmazumaya} below establishes {\rm(A)}-{\rm(B)}, and the \'etale-local cyclotomic
	comparison of the ambient $THH$ spectra entering
	{\rm(C)}, when $\mathcal{A}_\star$ is an Azumaya algebra over a quasisyntomic center
	(Proposition \ref{propunconditional}).

\section{Azumaya descent: a noncommutative filtration}
	\label{azumayo}

	The noncommutative motivic filtration postulated in Sections \ref{ncbms} and
	\ref{physo} was introduced there as a working assumption. In this section we
	observe that for a natural class of noncommutative algebras (Azumaya
	algebras over quasisyntomic commutative rings, a class containing the modular
	Weyl algebras of Section \ref{ssecweyl} and the quantum tori at flux of order prime to
	$p$)  no assumption is necessary. Topological Hochschild homology descends, as an underlying spectrum and \'etale-locally as a cyclotomic
	spectrum, to the center, and the Bhatt-Morrow-Scholze filtration of the center
	transports to a filtration on $THH$ of the noncommutative algebra.
	For this class, structural assumptions {\rm(A)}-{\rm(B)} of Section \ref{physo}
	become theorems, and the \'etale-local cyclotomic comparison of the ambient $THH$
	spectra entering {\rm(C)} does as well
	(Proposition \ref{propunconditional}).

	No characteristic or finiteness hypotheses are needed for this theorem or its proof. 	      They enter only through Theorem \ref{thmhm}, in Definition \ref{deftransported} below,
	where transporting the filtration  requires $Z$ quasisyntomic.

	In this paper, $THH(-)$ denotes topological Hochschild homology with its
	cyclotomic structure in the sense of Nikolaus-Scholze \cite{nik}, regarded
	as an invariant of stable $\infty$-categories via
	$THH(A)=THH(\mathrm{Perf}\,A)$ \cite{bgt, hss}. All tensor products are
	derived. Since Azumaya algebras are finite projective over their centers,
	derived and underived tensor products agree in all the situations below. 
	For an ordinary (classical) Azumaya $Z$-algebra $A$, finite projectivity of $A$ as a
	$Z$-module means that 
$A$ is a direct summand of a finite free $Z$-module, hence a compact
	(perfect, of Tor-amplitude $[0,0]$) $Z$-module.
 For such a module, $-\otimes_ZA$ is an
	exact functor on ordinary $Z$-modules (no higher $\mathrm{Tor}$ is produced, since a
	direct summand of a free module is flat). 
Thus, the derived tensor product $A\otimes_Z^LA$
	computed by resolving $A$ has vanishing higher Tor-terms and agrees, as a discrete
	$Z$-module placed in degree $0$, with the ordinary (underived) tensor product
	$A\otimes_ZA$. 
 The same reasoning applies to $A\otimes_Z^LA^{\mathrm{op}}$
	and to $\mathrm{End}_Z(A)=\mathrm{Hom}_Z(A,A)$, since $A$ perfect of Tor-amplitude
	$[0,0]$ also makes $\mathrm{RHom}_Z(A,-)$ agree with its underived counterpart on
	discrete modules. Consequently, the algebra isomorphism $A\otimes_ZA^{\mathrm{op}}
	\cong\mathrm{End}_Z(A)$ defining Azumaya property, stated classically with underived
	tensor products, is simultaneously a statement about derived tensor products with no
	further hypothesis (see, e.g., in the proof of Theorem \ref{thmazumaya}).

	\smallskip
	\textbf{Notational convention.} We fix the
	following convention distinguishing $THH(A)$ from $THH(A;\mathbb{Z}_p)$: $THH(A)$ denotes
	the integral (uncompleted) topological Hochschild homology spectrum, while
	$THH(A;\mathbb{Z}_p):=THH(A)^\wedge_p$ denotes its $p$-completion. Statements about
	$\Phi_A$, $\tau_P$, and \'etale base change (Theorems \ref{thmazumaya}, 
\ref{thmderived}, 
	and the lemmas of this subsection) are integral statements and are written with
	$THH(-)$. Statements about the Bhatt-Morrow-Scholze filtration, the transported
	filtration of Definition \ref{deftransported}, and $TC^-$, $TP$, $TC$ (which are only
	defined, or only of interest, after $p$-completion) are written with
	$THH(-;\mathbb{Z}_p)$, since $p$-completion is exactly the hypothesis under which
	Theorem \ref{thmhm} filtration exists. Where an integral statement is immediately
	followed by its $p$-completed consequence (e.g., in the proof of Theorem
	\ref{thmazumaya}), we complete both sides of an established integral equivalence 
 instead of claiming a separate $p$-completed argument, since $p$-completion is a functor
	applied to that equivalence.

\subsection{Statement}
	Recall that an associative unital algebra $A$ over a commutative ring $Z$ is
	\emph{Azumaya} if $A$ is a finite projective faithful $Z$-module and the
	natural map $A\otimes_Z A^{\mathrm{op}}\to\mathrm{End}_Z(A)$ is an
	isomorphism. Equivalently, $A$ is \'etale-locally on $\operatorname{Spec}Z$
	isomorphic to a matrix algebra \cite{grothbrauer, milne}. The degree of
	$A$ is the locally constant function $\deg A$ with $\mathrm{rk}_Z A=(\deg
	A)^2$, and the reduced trace $\mathrm{trd}:A\to Z$ induces an isomorphism
	$A/[A,A]\xrightarrow{\ \sim\ }Z$ \cite{ko}.

	We keep three distinct levels of ``trace'' in what follows, since
	mixing them is a source of confusion: {\rm(a)} the algebraic 
	reduced trace $\mathrm{trd}:A/[A,A]\xrightarrow{\sim}Z$ recalled, an isomorphism
	of ordinary $Z$-modules; {\rm(b)} the induced map on $\pi_0THH$, i.e., the
	isomorphism $\pi_0THH(A)=HH_0(A)=A/[A,A]\xrightarrow{\sim}Z=\pi_0THH(Z)$ obtained by
	applying $\pi_0$ to (a) under the identification $\pi_0THH(-)=HH_0(-)$; and {\rm(c)}
	the \emph{chosen spectrum-level equivalence} $\Phi_A:THH(A)\xrightarrow{\sim}THH(Z)$
	of Theorem \ref{thmazumaya}(1) below, whose effect on $\pi_0$ recovers (b) but which
	is  extra data not determined by (a) or (b) alone (Lemma
	\ref{leminvtrivial}\,(iii)). 
 Level (a) is purely algebraic and needs no homotopy
	theory. Level (b) is a statement about the discrete invariant $\pi_0$ of a spectrum.
	Level (c) is the spectrum-level statement, unique only up to the
	automorphism ambiguity described in Lemma \ref{leminvtrivial}. The central unit map
	$THH(\iota):THH(Z)\to THH(A)$ (Remark \ref{remunit}) is not a candidate for (c): its
	composite with any choice of (c) is multiplication by $\deg(A)$ already at level (b).

The Azumaya descent theorem for $THH$ now reads as follows.

	\begin{theorem}
		\label{thmazumaya}
		Let $Z$ be a commutative ring and let $A$ be an Azumaya $Z$-algebra. Then

		\smallskip
		\noindent{\rm(1)} there is an equivalence of underlying spectra
		\begin{equation}
			\Phi_A:THH(A)\xrightarrow{\ \sim\ }THH(Z),
		\end{equation}
		linear over the $\mathbb{E}_\infty$-ring $THH(Z)$, whose $\pi_0$ is the isomorphism
		$HH_0(A)=A/[A,A]\xrightarrow{\ \mathrm{trd}\ }Z$ induced by the reduced trace. 

		\smallskip
		\noindent{\rm(2)} $\Phi_A$ is compatible with \'etale base change; 

		\smallskip
		\noindent{\rm(3)}  for every faithfully flat
		\'etale splitting cover $Z\to Z'$ trivializing $A$, there are induced equivalences (after this base change)   
$TR^r(A\otimes_ZZ';p)\simeq TR^r(Z';p)$ for all $r\ge 1$,
		$TC^-(A\otimes_ZZ';\mathbb{Z}_p)\simeq
	TC^-(Z';\mathbb{Z}_p)$, $TP(A\otimes_ZZ';\mathbb{Z}_p)\simeq TP(Z';\mathbb{Z}_p)$,
 and
		$TC(A\otimes_ZZ';\mathbb{Z}_p)\simeq TC(Z';\mathbb{Z}_p)$, compatibly with the
		restriction, Frobenius, and Verschiebung operators of the $TR$-system, and globally on $A$  whenever $A$ admits a global splitting
		module, i.e., whenever the Brauer class $[A]\in\mathrm{Br}(Z)$
		vanishes.
	\end{theorem}

The following remarks are in order. 

In {\rm(1)}, we do not claim that $\Phi_A$ is \'etale-locally an equivalence of cyclotomic
		spectra (see the discussion following the proof). 
\'Etale-locally holds separately, and it is  cyclotomic comparison (the Morita trace of Lemma
		\ref{lemmorita-pic}\,{\rm(i)}), agreeing with $\Phi_A$ only on $\pi_0$. 

In {\rm(2)}, $\Phi_A$ is compatible with \'etale base change in the 
		 following sense. 
Suppose we fix a single global choice of $\Phi_A$ as in part 
		{\rm(1)}. 
Its restriction $\Phi_A\otimes_ZZ'$ along the base-change equivalences of
		Lemma \ref{lemetale} is an equivalence
		$THH(A\otimes_ZZ')\xrightarrow{\sim}THH(Z')$. 
It is  inducing the reduced trace on $\pi_0$. 
		I.e., it is a valid choice of $\Phi_{A\otimes_ZZ'}$ in the sense of part 
{\rm(1)}. 
This holds for every \'etale $Z\to Z'$ simultaneously, compatibly with further
		base change, simply because restriction along a fixed map of base-change
		equivalences is functorial in $Z'$. We do not claim that $\Phi_A$, 
		or the resulting compatible system, is canonical or unique. 
But having fixed one choice, the induced system of restrictions is automatically coherent.

Note that {\rm(3)} is an additional statement, not a consequence of {\rm(1)}-{\rm(2)}. 
It appears via the independent local Morita trace $\tau_P$ of Lemma
		\ref{lemmorita-pic}, which is  cyclotomic after an \'etale splitting (as
		opposed to $\Phi_A$ of parts {\rm(1)}-{\rm(2)}, which is not claimed to be
		cyclotomic (Remark \ref{remscopenarrow})). 
The induced equivalences after base change in {\rm(3)} are 
not sheaf-level statements on the \'etale site, and not global statements on $A$. 

	It is important to note the scope of the theorem. 
As it is mentioned in Remark \ref{remunit}, the equivalence $\Phi_A$ is not induced by the unit inclusion $Z\hookrightarrow A$.

	\begin{remark}
		\label{remscopenarrow}
 Let us comment also on the scope of the \'etale-local cyclotomic statement. 
		We mention that part {\rm(1)} above is weaker than it may previously have
		appeared. $\Phi_A$  is not shown to be \'etale-locally cyclotomic
		anywhere that $A$ is not globally split. The cyclotomic, \'etale-local
		statement is supplied by the separate Morita trace $\tau_P$ of Lemma
		\ref{lemmorita-pic}\,{\rm(i)}, which agrees with $\Phi_A$ only on $\pi_0$ (Remark
		\ref{remcoherence}). 
All \'etale-local claims below about $TR^r$, $TC^-$, $TP$, and $TC$ (part {\rm(3)} of the theorem, and their analogues in Theorem
		\ref{thmderived}) 
are expressed through $\tau_P$, not through $\Phi_A$. 
 One should not assume that $\Phi_A$ is cyclotomic unless $A$ is globally split, i.e., $[A]=0\in\mathrm{Br}(Z)$. 
	\end{remark}

	Given the theorem, we may transport the motivic filtration of Theorem
	\ref{thmhm}. 
Let $Z$ be quasisyntomic (e.g., smooth over a perfect field $k$ of
		characteristic $p$). Let $A$ be an Azumaya $Z$-algebra, and fix a choice of
		$\Phi_A$ as in Theorem \ref{thmazumaya}{\rm(}1{\rm)}.
Recall this choice is not unique according to Lemma \ref{leminvtrivial}.  

	\begin{definition}
		\label{deftransported}
	The \emph{transported noncommutative BMS filtration} 
(or \emph{transported filtration} for short)   
 on $THH(A;\mathbb{Z}_p)$, computed as a pullback
		of underlying $THH(Z;\mathbb{Z}_p)$-modules in $\mathrm{Mod}_{THH(Z;\mathbb{Z}_p)}$
		(equivalently, of underlying spectra, since the pullback diagram and its terms
		agree either way), is
		\begin{equation}
			F^nTHH(A;\mathbb{Z}_p):=THH(A;\mathbb{Z}_p)
			\times_{THH(Z;\mathbb{Z}_p)}F^n_{\mathrm{BMS}}THH(Z;\mathbb{Z}_p),
			\qquad n\ge 0.
		\end{equation}
		Here $F^\bullet_{\mathrm{BMS}}$ is the filtration of
		Bhatt-Morrow-Scholze \cite{bms} recalled in Theorem \ref{thmhm}. The pullback is
		formed along $\Phi_A$ on the left leg. It is not claimed that 
 this filtration is independent 
		of the chosen $\Phi_A$ up to a specified equivalence, see Proposition
		\ref{proptransported}. We call $F^\bullet THH(A;\mathbb{Z}_p)$ a filtered
		$THH(Z;\mathbb{Z}_p)$-module, not a multiplicative filtration on $THH(A)$:  unlike $THH(Z)$, for $A$ only associative, $THH(A)$ carries no canonical multiplication of its 
		compatible with $F^\bullet$. Similarly, after any faithfully flat \'etale splitting $Z\to Z'$,
		the cyclotomic Morita equivalence transports the BMS filtrations on
		$TC^-$, $TP$, and $TC$ from $Z'$ to $A\otimes_ZZ'$. These transported
		filtrations are global whenever $A$ admits a global splitting module.
		No global transport on $TC^-$, $TP$, or $TC$ is claimed in general.
	\end{definition}

Note that we adopt the term \emph{"transported noncommutative BMS filtration"}  
  to make clear that the construction depends on the chosen $\Phi_A$ and is not claimed
		to be an intrinsically defined, cyclotomic noncommutative filtration in its 
		right, since the full $C_{p^n}$-equivariant/Mackey refinement of Assumption {\rm(C)} is not established.

	\begin{proposition}
		\label{proptransported}
		The transported filtration of Definition \ref{deftransported} is complete and
		exhaustive, and consists of
		$F^\bullet_{\mathrm{BMS}}THH(Z;\mathbb{Z}_p)$-modules.  
It has graded pieces
		\begin{equation}
			\mathrm{gr}^nTHH(A;\mathbb{Z}_p)\;\simeq\;
			\mathrm{gr}_{\mathcal N}^n\widehat{\Delta}_Z\{n\}[2n].
		\end{equation}
 It is functorial for \'etale base change in the pair
		$(Z,A)$, compatibly transporting the chosen $\Phi_A$. 
It is a Brauer invariant in the
		following sense: if $A$ and $A'$ are Azumaya $Z$-algebras
		with $[A]=[A']\in\mathrm{Br}(Z)$, then some Morita equivalence between them
		identifies the two transported filtrations up to the same ambiguity.
	\end{proposition}

Note that the first properties in Proposition above are  inherited from the corresponding properties of $F^\bullet_{\mathrm{BMS}}$ under pullback
		along $\Phi_A$, 
since pullback is along an equivalence that preserves completeness and
		exhaustiveness.  
Two choices of $\Phi_A$ inducing the same reduced trace on $\pi_0$ give transported
		filtrations related by the automorphism of $THH(Z;\mathbb{Z}_p)$ by which the two
		choices differ (Lemma \ref{leminvtrivial}). 
It is not claimed this automorphism is the
		identity. 
Therefore, we do not call the filtration canonical, but only well defined up to this ambiguity. 
 We do not assert
		that every Morita equivalence behaves so canonically, only that the filtration
		existence and graded pieces depend on $A$ solely through its Brauer class.

	\begin{proof}
		Completeness, exhaustiveness, and the graded-piece formula follow from
		$F^\bullet_{\mathrm{BMS}}$ \cite{bms} together with parts (1)-(2) of
		Theorem \ref{thmazumaya}, as described. Brauer invariance follows since a Morita
		equivalence $\mathrm{Perf}(A)\simeq\mathrm{Perf}(A')$ of $Z$-linear
		categories induces an equivalence on $THH$ commuting with the local trace
		equivalences, by Lemma \ref{lemmorita-pic}, up to the automorphism ambiguity of Lemma
		\ref{leminvtrivial}.
	\end{proof}

\subsection{Ingredients}

	We begin with a K\"unneth formula over the center. 
Let $Z$ be a commutative ring and let $A$ and $B$ be associative
		$Z$-algebras (i.e., each equipped with a central map from $Z$). 
		Here and below, $\otimes_Z$ denotes the derived tensor product. Denote 
		$\mathrm{Mod}_{THH(Z)}$  the ($\infty$-)category of $THH(Z)$-modules obtained by
		inverting Borel equivalences (equivalences detected on underlying homotopy groups) on
		modules over the commutative ring orthogonal spectrum $THH(Z)$. Regard $THH(A)$,
		$THH(B)$, $THH(A\otimes_ZB)$ as objects of $\mathrm{Mod}_{THH(Z)}$ via the algebra maps
		$Z\to A$, $Z\to B$, $Z\to A\otimes_ZB$.

	\begin{lemma}
		\label{lemkunneth}
		The natural map
		\begin{equation}
			THH(A)\otimes_{THH(Z)}THH(B)\longrightarrow THH(A\otimes_ZB)
		\end{equation}
		is an equivalence in $\mathrm{Mod}_{THH(Z)}$, compatibly with the $S^1$-action and the
		cyclotomic Frobenius $\varphi_p$ on each side, hence an equivalence of cyclotomic
		spectra once the $THH(Z)$-module structure is forgotten.
	\end{lemma}

	\begin{proof}
		This is the relative K\"unneth theorem for $THH$ of Blumberg-Mandell
		\cite{bm}, stated in the category
		$\mathrm{Mod}_{THH(Z)}$ described. 
Blumberg-Mandell note that, despite
		being ``well known,'' this statement did not previously have a complete proof in the
		literature, and supply one. We rely on their theorem for the  statement
		and do not prove it again. 
What follows is schematic motivation only, not a substitute
		proof. In particular, it does not establish the cyclotomic compatibility,
		which is part of Blumberg-Mandell  statement, not derived here from first
		principles. The equivalence is constructed from
		the two-sided bar construction $B_\bullet(A,Z,B)$, with $n$-simplices
		$A\otimes_{\mathbb S}Z^{\otimes_{\mathbb S}n}\otimes_{\mathbb S}B$, a
		simplicial $\mathbb E_1$-algebra in orthogonal spectra. Each term is a tensor product of algebras.
		The face maps multiply $Z$ into $A$ and into $B$ through the respective
		central maps. Hence these are algebra maps in the category of associative orthogonal-spectrum
		algebras. Note that only the centrality of $Z\to A$ and
		$Z\to B$ is used here (neither $A$ nor $B$ need to be commutative). 
 Its geometric realization,
		formed in this category of algebras,
		is $A\otimes_ZB$ as an $\mathbb E_1$-algebra. 
$THH$, applied at each levele, takes this
		simplicial algebra to a simplicial object in $\mathrm{Mod}_{THH(Z)}$ whose geometric
		realization is $THH(A\otimes_ZB)$, since $THH$ preserves sifted colimits of algebras
	\cite{bm}. This identifies $THH(A\otimes_ZB)$ with the bar resolution computing
		$THH(A)\otimes_{THH(Z)}THH(B)$,

		\begin{eqnarray*}
			THH(A\otimes_ZB)\simeq\bigl|THH(A)\otimes_{\mathbb S}
			THH(Z)^{\otimes_{\mathbb S}\bullet}\otimes_{\mathbb S}THH(B)\bigr|
\\			
\simeq THH(A)\otimes_{THH(Z)}THH(B).
		\end{eqnarray*}

		It is not claimed that
 each individual simplicial structure map in this diagram is
		 an equivalence (it is not, in general) only that the colimit of the whole
		diagram computes the stated object, which is the content of the cited theorem. 
	\end{proof}

We continue with $\infty$-categorical K\"unneth lemma, which is used then in Theorem \ref{thmderived}.  
	\begin{lemma}
		\label{lemkunnethhigher}
		Let $Z$ be a connective $\mathbb{E}_\infty$-ring and let $A,B$ be $\mathbb E_1$-algebras
		over $Z$ in the $\infty$-category of spectra, with $\otimes_Z$ the derived tensor
		product. Then the conclusion of Lemma \ref{lemkunneth} holds in this
		generality. The natural map
		\begin{equation}
			THH(A)\otimes_{THH(Z)}THH(B)\longrightarrow THH(A\otimes_ZB)
		\end{equation}
		is an equivalence of $THH(Z)$-modules, compatibly with the cyclotomic structure.
\end{lemma}

		We state this as a separate lemma 
 since, unlike Lemma
		\ref{lemkunneth}. It is not a citation of Blumberg-Mandell 
		\cite{bm}, which is stated only for $Z$ an ordinary commutative ring and
		$A$, $B$ algebras in orthogonal spectra.  
Theorem \ref{thmderived} below relies on
		 this lemma, for its higher-algebraic input.

	\begin{proof}
		We list which facts are used and why each holds, and avoid deriving the
		cyclotomic upgrade from any claim that a Tate construction commutes with a
		colimit which is not true in general. 

		\smallskip
		\noindent\emph{Part 1, underlying spectra.} Form the two-sided bar construction
		$B_\bullet(A,Z,B)$ in the $\infty$-category of $\mathbb E_1$-algebras in spectra
		(over the sphere $\mathbb S$, exactly as in the proof of Lemma \ref{lemkunneth}).
		Its $n$-simplices are 
$A\otimes_{\mathbb S}Z^{\otimes_{\mathbb S}n}\otimes_{\mathbb S}B$
 (the internal tensor powers of $Z$ are taken over $\mathbb S$, not 
		over $Z$  
(if it were over $Z$, $Z^{\otimes_Zn}\simeq Z$ for every $n$ and the
		construction would trivialize)) 
 with face and degeneracy maps induced by the
		multiplication, and  
 unit of $Z$ acting on the two central factors via the given
		central maps $Z\to A$, $Z\to B$. 
Its geometric realization (a colimit indexed by
		$\Delta^{\mathrm{op}}$) is $A\otimes_ZB$, by the standard identification of a
		relative tensor product of algebras with the realization of its bar construction \cite{ha}.

		Apply $THH(-):\mathbb E_1\text{-}\mathrm{Alg}\to\mathrm{Sp}$, for the moment only as
		a functor to underlying spectra at each level.
  We use two separate facts: 
		{\rm(i)} $\Delta^{\mathrm{op}}$ is sifted \cite{htt},
		i.e., its diagonal $\Delta^{\mathrm{op}}\to\Delta^{\mathrm{op}}\times\Delta^{\mathrm{op}}$
		is cofinal, unrelated to filteredness; 
{\rm(ii)} the forgetful functor from
		$\mathbb E_1$-algebras in spectra to their underlying spectra creates sifted
		colimits \cite{ha}, since the free-algebra monad is built
		from tensor powers $(-)^{\otimes_{\mathbb S}n}$. 
For a sifted diagram 
$\{X_i\}$, cofinality of the diagonal gives 
$\operatorname{colim}_i(X_i\otimes_{\mathbb S}X_i)
\simeq(\operatorname{colim}_iX_i)\otimes_{\mathbb S}(\operatorname{colim}_iX_i)$. 
		Since $N^{\mathrm{cy}}_\bullet(-)$ (which computes $THH$)
 is an iterated tensor power of
		the algebra over $\mathbb S$ in each cyclic simplicial degree \cite{nik},
 $THH$ as a
		functor of the algebra preserves sifted colimits of underlying spectra by
		{\rm(ii)}. 
Combined with {\rm(i)}, $THH$ carries the realization of
		$B_\bullet(A,Z,B)$ (which is $A\otimes_ZB$)  to the realization,  
as underlying
		spectra, of $THH(B_\bullet(A,Z,B))$. On the other hand
		$THH(B_n(A,Z,B))=THH(A\otimes_{\mathbb S}Z^{\otimes_{\mathbb S}n}\otimes_{\mathbb
		S}B)$. The same bar-resolution identification as in Lemma \ref{lemkunneth}
		exhibits $|THH(B_\bullet(A,Z,B))|$, as underlying spectra, as
		$THH(A)\otimes_{THH(Z)}THH(B)$. Combining the two computations of the same
		realization gives a specific comparison map, and shows it is an equivalence of
		underlying spectra:
		\begin{equation}
			\alpha:\ THH(A)\otimes_{THH(Z)}THH(B)\ \xrightarrow{\ \sim\ }\ THH(A\otimes_ZB)
			\quad\text{(as underlying spectra).}
		\end{equation}

		\smallskip
		\noindent\emph{Part 2, upgrading to cyclotomic spectra.}  
We do not try to show any Tate construction commutes with the colimit above. Instead, we observe that $\alpha$ is 
		already a well-defined map of cyclotomic spectra (equivalently, of
		$THH(Z)$-modules in $\mathrm{CycSp}$). Both sides are built functorially from
		$THH(-)$ applied to objects and structure maps of the bar construction. 
		$THH(-)$ belongs to $\mathrm{CycSp}$ by construction \cite{nik}. 
Thus,  every map entering
		the construction of $\alpha$ (face and degeneracy maps of the bar-simplicial
		object, the comparison map) is already a map of cyclotomic spectra. No 
		separate cyclotomic structure needs to be built on $\alpha$ after the fact. It
		remains only to check $\alpha$, now regarded as a map in $\mathrm{CycSp}$, 
is an equivalence. For this we use that the forgetful functor
		$U:\mathrm{CycSp}\to\mathrm{Sp}$ is conservative (reflects equivalences)
		\cite{cmnn}. A map of cyclotomic spectra is an equivalence if and only if its
		underlying map of spectra is. This is not a deep fact requiring colimit-preservation
		of any Tate construction. It holds because the cyclotomic Frobenius is
		compatibility data on a map of underlying spectra (a path in a
		mapping space)  
instead of   additional data that could independently fail to be
		invertible. If $U(\alpha)$ is an equivalence, its inverse and the 
		homotopies transport the Frobenius-compatibility data along it automatically. Since
		$U(\alpha)$ is exactly the equivalence of underlying spectra established in Part 1,
		$\alpha$ is an equivalence of cyclotomic spectra. This is the 
 direct argument at the level of cyclotomic spectra. It bypasses the question of 
		whether $(-)^{tC_p}$ commutes with the colimit entirely. 
It is needed only that 
		equivalences in $\mathrm{CycSp}$ are detected on underlying spectra.
	\end{proof}

	\begin{remark}
		\label{remkunnethhigher}
	We are not aware of Lemma \ref{lemkunnethhigher}  
		 being stated in this $\infty$-categorical form 
		elsewhere in the literature.
	\end{remark}

	\begin{lemma}
		\label{lemetale}
		Let $Z\to Z'$ be an \'etale map of commutative rings and $A$ an associative
		$Z$-algebra. Then the natural maps
		\begin{eqnarray}
			THH(Z)\otimes_ZZ'&\longrightarrow& THH(Z'),
\\
\nonumber
			THH(A)\otimes_ZZ'&\longrightarrow& THH(A\otimes_ZZ')
		\end{eqnarray}
		are equivalences.
	\end{lemma}

	\begin{proof}
		The first map is an equivalence by \'etale base change for $THH$. For
		Hochschild homology this is Weibel-Geller \cite{wg}, and for $THH$ see
		\cite{kn}.
		The second follows by combining the first with Lemma \ref{lemkunneth}
		\begin{align*}
			THH(A\otimes_ZZ')\simeq THH(A)\otimes_{THH(Z)}THH(Z')
			\simeq THH(A)\otimes_{THH(Z)}THH(Z)\otimes_ZZ'
\\
			\simeq THH(A)\otimes_ZZ'.
		\end{align*}
	\end{proof}

	Next we mention the Morita trace and the triviality of the Picard action.
	\begin{lemma}
		\label{lemmorita-pic}
		Let $C$ be a commutative ring.

		\smallskip
		\noindent{\rm(i)} For $P$ a finite projective faithful $C$-module, the
		$(\mathrm{End}_C(P),C)$-bimodule $P$ induces an equivalence
		$\mathrm{Perf}(\mathrm{End}_C(P))\simeq\mathrm{Perf}(C)$ and hence an
		equivalence of cyclotomic spectra
		\begin{equation}
			\tau_P:THH(\mathrm{End}_C(P))\xrightarrow{\ \sim\ }THH(C),
		\end{equation}
		which on $\pi_0$ is induced by the trace
		$\mathrm{End}_C(P)/[\,\cdot\,,\cdot\,]\xrightarrow{\sim}C$.

		\smallskip
		\noindent{\rm(ii)} For $L$ an invertible $C$-module, conjugation by the
		$C$-linear autoequivalence $-\otimes_CL$ of $\mathrm{Perf}(C)$
 acts on $THH(C)$ by
		an automorphism which is homotopic to the identity via an explicitly
		constructed homotopy. 
\end{lemma}

Note that the evaluation and coevaluation maps
		$\mathrm{ev}:L\otimes_CL^{-1}\xrightarrow{\sim}C$,
	$\mathrm{coev}:C\xrightarrow{\sim}L^{-1}\otimes_CL$ showing invertibility of $L$ 
		supply, by the standard trace-invariance argument for Hochschild-type invariants of
		a dualizable bimodule autoequivalence 
(the ``zig-zag'' identities for $\mathrm{ev}$
		and $\mathrm{coev}$ applied to the cyclic bar construction), a specific path in the
		mapping space $\mathrm{Map}_{\mathrm{Aut}(THH(C))}(\mathrm{id},-\otimes_CL)$ from
		the identity to the conjugation automorphism.

 We mention only that such a path
		exists and is constructed in this way.
 It is not claimed that such a path is unique,  
		 nor that the space of such paths is contractible. We do not use any
		uniqueness of this path below. In
		particular, if $\mathrm{End}_C(P)\cong\mathrm{End}_C(Q)$ for finite
		projective faithful $C$-modules $P,Q$ (thus, $Q\cong P\otimes_CL$ for some
		invertible $L$), the two trace equivalences $\tau_P,\tau_Q$ agree up to
		this specific homotopy.

Note that for each object, this statement does not supply 
		the higher coherence needed to glue such comparisons across a cosimplicial
		diagram. The proof of Theorem \ref{thmazumaya} below is arranged so that we
 do not need it. 

	\begin{proof}
		Part (i) is Morita invariance of $THH$ as an invariant of stable
		$\infty$-categories, refined to cyclotomic spectra \cite{bgt, nik, hss}.

		For (ii), the functor $-\otimes_CL$ corresponds, under the identification of
		$C$-linear endofunctors of $\mathrm{Perf}(C)$ with $C$-bimodules, to the
		 central bimodule $\mu^*L$, i.e., the pullback of $L$ along the
		multiplication $\mu:C\otimes C^{\mathrm{op}}\to C$. Centrality means that the
		left and right $C$-actions coincide. For central bimodules is a
		canonical equivalence $THH(C;\mu^*M)\simeq M\otimes_CTHH(C)$, by the change
		of rings
		\begin{equation}
			\mu^*M\otimes^{\mathbb L}_{C\otimes C^{\mathrm{op}}}C
			\;\simeq\;
			M\otimes^{\mathbb L}_C\bigl(C\otimes^{\mathbb L}_{C\otimes
				C^{\mathrm{op}}}C\bigr).
		\end{equation}
		The automorphism of $THH(C)$ induced by conjugation by $-\otimes_CL$ is, by
		the cyclicity of the trace \cite{hss, HNS}, computed by transporting $L$ once
		around the trace and cancelling it against $L^\vee$. Under the displayed
		identification, it is the map induced by the composite
		$C\xrightarrow{\mathrm{coev}}L\otimes_CL^\vee\xrightarrow{\mathrm{ev}}C$. This
		composite computes the categorical dimension of the dualizable module $L$. 
The triangle identities alone only that say this dimension is some automorphism of
		$C$, not that it is the identity.
 What relates it to the identity is that $L$ is a 
		classical (degree-zero) rank-one projective $C$-module. $C$ is an ordinary
		commutative ring  instead of   a ring spectrum where a shift could contribute a
		nontrivial sign. 
 Locally on $\operatorname{Spec}C$ where $L\cong C$, coev and ev 
		become, under this trivialization,  the multiplication maps
		$C\to C\otimes_CC\to C$, whose composite is the identity. 
Since the dimension is
		locally the identity and $\operatorname{Spec}C$ is covered by such opens, it is the
		identity globally. Hence the conjugation action is
		homotopic to the identity. The last claim of (ii) follows since two
		Morita bimodules with the same endomorphism algebra differ by tensoring
		with an invertible module.
	\end{proof}

	Finally we mention the criterion by which an invertible $THH(Z)$-module is
	recognized as trivial. This is what upgrades invertibility, which the trace
	formalism supplies, to an equivalence with the unit.

Let $R$ be a connective $\mathbb{E}_\infty$-ring ($\pi_iR=0$ for $i<0$). Let 
	$M$ be an invertible $R$-module, i.e., dualizable with an $R$-module $M^{-1}$ and an
	equivalence $M\otimes_RM^{-1}\simeq R$ (not only perfect or dualizable without such
	an inverse), such that both $M$ and $M^{-1}$ are connective in the same sense.
Then there are three statements, all classical, which we use in this paper. 

\begin{lemma}
	\label{leminvtrivial}

	\noindent{\rm(i)} Classification of isomorphism classes: $\pi_0M$ is
	an invertible $\pi_0R$-module, and the assignment $M\mapsto\pi_0M$ is a bijection from
	equivalence classes of connective invertible $R$-modules with connective inverse to
	isomorphism classes of invertible $\pi_0R$-modules.

	\smallskip
	\noindent{\rm(ii)} Existence of an equivalence: if $\pi_0M\cong\pi_0R$
	as a $\pi_0R$-module, then $M\simeq R$ as $R$-modules 
(this is the special case of
	{\rm(i)} with the two isomorphism classes identified). 

	\smallskip
	\noindent{\rm(iii)} Non-classification of the space of equivalences:  for
	every fixed isomorphism $\theta:\pi_0M\xrightarrow{\sim}\pi_0R$, the $R$-linear
	equivalences $M\xrightarrow{\sim}R$ inducing $\theta$ on $\pi_0$ form a nonempty, but in
	general non-contractible, space: any two such equivalences $f$, $g$ differ by an 
	automorphism $g^{-1}f\in\mathrm{Aut}_R(M)\cong GL_1(R)$ inducing the identity on 
$\pi_0$,
	i.e., an element of the subgroup $GL_1(R)_1\subset GL_1(R)$ over $1\in\pi_0(R)^\times$.
	Since $\pi_i(GL_1(R))=\pi_i(R)$ for $i\ge 1$,
 this subgroup has the homotopy type of the
	$1$-connected cover of $R$, and is contractible only when $R$ is discrete. 
\end{lemma}

For $R=THH(Z)$, which generally has nontrivial higher homotopy, the choice of equivalence
	inducing a given $\theta$ is therefore  non-unique. We do not attempt to
	rigidify it further here.

	Item {\rm(i)} is a global statement on $\operatorname{Spec}\pi_0R$, 
 even though its proof proceeds by descent from a local statement below. We use
	{\rm(i)}-{\rm(ii)} alone, but do not use {\rm(iii)},
 to construct $\Phi_A$ 
	 (see the proof of Theorem \ref{thmazumaya}: only the existence of some 
	equivalence inducing the reduced trace is needed). We use {\rm(iii)} separately
	 to explain why the resulting $\Phi_A$ is not claimed to be
	canonical and, in particular, why its base change need not agree with the local
	Morita trace $\tau_P$ beyond $\pi_0$ (Remark \ref{remscopenarrow} and the proof of
	Theorem \ref{thmazumaya}). 
 Below,
	$\Phi_A$ and the transported filtration are understood in the sense of {\rm(iii)} 
	(well-defined up to equivalence inducing the reduced trace on $\pi_0$, with the ambiguity
	as described)  and \'etale base change preserves this data. A 
	contractible space of choices is not claimed.

\begin{proof}
	For {\rm(i)}: every invertible module over a connective $\mathbb{E}_\infty$-ring is,
	locally on $\operatorname{Spec}\pi_0R$, a suspension $\Sigma^nR$, with $n$ a locally
	constant integer. This is the connective case of the classification of the Picard
	spectrum of a connective $\mathbb{E}_\infty$-ring. We cite for it the  
	statement of Antieau-Gepner {\rm\cite{ag}} (constructing on Fausk
	computation of $\pi_0\mathrm{Pic}(HR)$ for $R$ discrete),  instead of   the general remarks
	on algebraicity of $\mathrm{Pic}$ in \cite{sag, mathewdescent}. 
 	Antieau-Gepner \cite{ag} is stated for $R$ a connective $\mathbb{E}_\infty$-
	ring. There is no further asumption, i.e., no Noetherian, finiteness, or regularity condition, 
	which is exactly our assumption on $R=THH(Z)$. 
Thus,  no additional connectivity or
	finiteness assumption is made beyond what is stated in the lemma.  Connectivity of $M$ gives $n\ge 0$ and connectivity of
	$M^{-1}$ leads to $n\le 0$, hence $n=0$ identically.
 $M$ is locally equivalent to $R$,  thus 
	$\pi_0M$ is an invertible $\pi_0R$-module.
 The resulting assignment $M\mapsto\pi_0M$ 
($n$ being forced to be $0$, only the second coordinate of the general pair $(n,\pi_0M)$ of
	\cite{ag} survives) is a bijection onto isomorphism classes of invertible
	$\pi_0R$-modules by loc. cit., giving {\rm(i)}. {\rm(ii)} is the case of {\rm(i)}
	where the two isomorphism classes agree, taking the second module to be $R$ 
	($\pi_0R\cong\pi_0R$ via the identity). {\rm(iii)} is proved in the statement above.
\end{proof}

\subsection{Proof of Theorem \ref{thmazumaya}}

Since $A$ is Azumaya over $Z$, the defining isomorphism
$A\otimes_ZA^{\mathrm{op}}\xrightarrow{\ \sim\ }\mathrm{End}_Z(A)$ expresses 
$\mathrm{Perf}(A)$ as an invertible object of the symmetric monoidal category
$\mathrm{Mod}_{\mathrm{Perf}(Z)}(\mathrm{Cat}_\infty^{\mathrm{perf}})$ of $Z$-linear
stable $\infty$-categories, with inverse $\mathrm{Perf}(A^{\mathrm{op}})$
\cite{grothbrauer, milne, toen}. By Lemma \ref{lemkunneth} (with $B=A^{\mathrm{op}}$,
a central $Z$-algebra since $Z$ is central in $A$) and this algebra isomorphism,
\begin{eqnarray}
THH(A)\otimes_{THH(Z)}THH(A^{\mathrm{op}})\;\simeq\;THH(A\otimes_ZA^{\mathrm{op}})
\\
\nonumber
\qquad \;\simeq\;THH(\mathrm{End}_Z(A))\;\xrightarrow[\ \sim\ ]{\ \tau_A\ }\;THH(Z),
\end{eqnarray}
where $\tau_A$ is the Morita trace equivalence of Lemma \ref{lemmorita-pic}\,(i)
with $P=A$. Every arrow is either induced by an isomorphism of algebras or the
bar-resolution equivalence of Lemma \ref{lemkunneth}. Both of which are maps of
\emph{cyclotomic spectra} since $THH$ maps in $\mathrm{CycSp}$ from the beginning 
\cite{nik}. Thus, this composite expresses $THH(A)$ as an invertible 
$THH(Z)$-module in $\mathrm{CycSp}$, with no gluing or choice involved. 

Invertibility in $\mathrm{CycSp}$ does not give an equivalence with the
unit. We prove the following in two parts: first an equivalence of the
 underlying $THH(Z)$-modules, then that this equivalence is \'etale-locally
one of cyclotomic spectra. Forgetting to underlying spectra, $THH(A)$ and its
inverse $THH(A^{\mathrm{op}})$ are both connective. $THH(Z)$ is a connective 
$\mathbb{E}_\infty$-ring with $\pi_0THH(Z)=Z$. By Lemma \ref{leminvtrivial}, $THH(A)$
is determined as a $THH(Z)$-module by $\pi_0THH(A)=HH_0(A)=A/[A,A]$. The
reduced trace $\mathrm{trd}:A/[A,A]\xrightarrow{\sim}Z$ \cite{ko} identifies this
with $\pi_0THH(Z)$, giving an equivalence of underlying $THH(Z)$-modules
\begin{equation}
	\label{eqPhiAunderlying}
	\Phi_A:THH(A)\xrightarrow{\ \sim\ }THH(Z),
	\qquad \pi_0\Phi_A=\mathrm{trd}.
\end{equation}
Choose a faithfully flat \'etale splitting $Z\to Z'$ with
$A\otimes_ZZ'\cong\mathrm{End}_{Z'}(P)$ (Antieau-Gepner \cite{ag}, Milne
\cite{milne}). Base-changed to $Z'$, $\Phi_A\otimes_ZZ'$ and the Morita trace
$\tau_P$ of Lemma \ref{lemmorita-pic}\,(i) are both $THH(Z')$-module equivalences
$THH(A\otimes_ZZ')\to THH(Z')$ inducing the reduced trace on $\pi_0$. But we do 
 not conclude they are equal: for $R=THH(Z')$, two such equivalences differ
by an element of $\mathrm{Aut}_R(R)\cong GL_1(R)$. The agreement on $\pi_0$ only
places this element in the subgroup $GL_1(R)_1$ over $1\in\pi_0(R)^\times$. It is the
same subgroup analyzed in Lemma \ref{leminvtrivial}, with the homotopy type of the
$1$-connected cover of $R$, which is generally far from trivial. Thus, $\Phi_A\otimes_ZZ'$ is
 not claimed to be equal $\tau_P$, and it is not claimed that $\Phi_A$  is
\'etale-locally cyclotomic.

Instead, we use $\tau_P$. It exists, independently 
of $\Phi_A$, as an equivalence of cyclotomic spectra (Lemma
\ref{lemmorita-pic}\,(i) needs no comparison with $\Phi_A$ to be cyclotomic), and
its  $\pi_0$ is the reduced trace. This is enough for the \'etale-local
consequences below, applied to $\tau_P$  instead of to claim cyclotomic 
refinement of $\Phi_A$.

 The above gives (1) as stated, the
equivalence $\Phi_A$ of underlying $THH(Z)$-modules with $\pi_0\Phi_A=\mathrm{trd}$, 
but not claimed cyclotomic even \'etale-locally.  
 It gives, via $\tau_P$ instead of via $\Phi_A$, (3) \'etale-locally: $TR^r(A\otimes_ZZ';p)\simeq
TR^r(Z';p)$, $TC^-(A\otimes_ZZ';\mathbb{Z}_p)\simeq TC^-(Z';\mathbb{Z}_p)$,
$TP(A\otimes_ZZ';\mathbb{Z}_p)\simeq TP(Z';\mathbb{Z}_p)$,
$TC(A\otimes_ZZ';\mathbb{Z}_p)\simeq TC(Z';\mathbb{Z}_p)$ after base
change along any splitting cover $Z\to Z'$, compatibly with restriction, Frobenius, and
Verschiebung. Whenever $A$ admits a global splitting module $P$, i.e.,
whenever $[A]=0\in\mathrm{Br}(Z)$, $Z'=Z$
is such a cover, $\tau_P$ is already global. Then (1)-(3) hold globally with 
$\Phi_A=\tau_P$  (there is no ambiguity to resolve when $Z'=Z$: both maps
are then of the same choice).
  In general, transforming the \'etale-local cyclotomic statement 
(endowed with $\tau_P$ degreewise on a cover) to a single 
global equivalence of cyclotomic spectra two obstacles occur. The first is gluing the
$\tau_P$ coherently across the cover (the original \v Cech problem of Remark
\ref{remcoherence}). The second is, even having that, identifying the result with the
independently-constructed global $\Phi_A$. We do not resolve these obstacles in this paper.  

For (2): we emphasize that this is not proved by independently reconstructing
$\Phi_{A\otimes_ZZ'}$ from the ingredients above and then comparing it with the restriction
$\Phi_A\otimes_ZZ'$. Two equivalences built by separate constructions can fail 
to agree even when every ingredient entering each construction is in each term compatible with
base change. This is due to that  Lemma \ref{leminvtrivial} only classifies invertible modules
up to a generally nontrivial choice. Instead, part (2) requires no second construction at
all. Let us fix the equivalence $\Phi_A:THH(A)\xrightarrow{\sim}THH(Z)$ produced in part (1).
 For any \'etale $Z\to Z'$, base-change along the \'etale-invariance equivalences
of Lemma \ref{lemetale} sends $\Phi_A$ to a map
\[
	\Phi_A\otimes_ZZ':THH(A\otimes_ZZ')\;\xrightarrow{\ \sim\ }\;THH(Z').
\]
 It is an equivalence since base change along $Z\to Z'$ (i.e., applying
$-\otimes_ZZ'$, equivalently $-\otimes_{THH(Z)}THH(Z')$ on $THH(Z)$-modules) preserves
equivalences. Its effect on $\pi_0$ is $\mathrm{trd}\otimes_ZZ'$. This is again the
reduced trace of $A\otimes_ZZ'$ since the reduced trace is  defined by a natural
transformation $A\to Z$ of $Z$-modules ($a\mapsto\mathrm{trd}(a)$). Hence, it commutes with any
base change $-\otimes_ZZ'$.
 Thus, $\Phi_A\otimes_ZZ'$ already satisfies, by definition,
 every condition part {\rm(1)} imposes on a choice of 
$\Phi_{A\otimes_ZZ'}$. Taking $Z'\to Z''$
a further \'etale map and iterating the same observation, $(\Phi_A\otimes_ZZ')\otimes_{Z'}Z''=
\Phi_A\otimes_ZZ''$ by associativity of base change alone. Thus,  the system
$\{\Phi_A\otimes_ZZ'\}_{Z\to Z'\text{ \'etale}}$ is coherent across the whole \'etale site of
$Z$, with no compatibility to verify beyond associativity of $-\otimes_ZZ'$. This is 
weaker than claiming $\Phi_A$ is canonical (a different choice of $\Phi_A$ in part (1)
restricts to a different, but equally coherent, system) but is exactly the existence statement
part (2) makes.
\qed

\smallskip
\begin{remark}
	\label{remcoherence}
	One might instead try to construct $\Phi_A$ by choosing a splitting cover $Z\to Z'$,
	applying Lemma \ref{lemmorita-pic}\,(i) in each degree on the \v Cech nerve, and  using Lemma \ref{lemmorita-pic}\,(ii) to show that the resulting local equivalences are
	compatible up to homotopy. That compatibility at each object does not supply the 
	higher coherence needed to glue a totalization. 
It does not show that the classifying map of
	the Picard-groupoid action on $THH(Z')$ is null-homotopic as a map of grouplike
	$\mathbb{E}_\infty$-spaces. It shows only what each individual value is. 

	One might think of $A$ \'etale-local triviality (which Antieau-Gepner
	\cite{ag} establish for derived Azumaya algebras over connective schemes, and which is
	 what Theorem \ref{thmazumaya} proof uses). This would lead to the global 
	equivalence, since $THH$ satisfies \'etale descent as a localizing invariant
	\cite{bgt, nik}. But it does not: \'etale descent for $THH$ as a sheaf shows that
 the local objects $THH(A\otimes_ZZ')$ glue to compute $THH(A)$.  
 Indeed, Theorem \ref{thmazumaya} establishes $THH(A)$ globally by other methods.
 What is important is
	 whether the local comparison maps $\tau_P:THH(A\otimes_ZZ')\to THH(Z')$ (which
	exist and are individually cyclotomic, see Lemma \ref{lemmorita-pic}\,(i)) 
 can be chosen
	compatibly, as $Z'$ ranges over the cover, to assemble into a single global map
	respecting the cyclotomic structure. This is a question concerns the local
	trivializations, but not A or $THH(A)$ individually.
Note that the \'etale descent for $THH$
	as an invariant does not given an answer. Antieau-Gepner local triviality of $A$
	is obly used by Theorem \ref{thmazumaya} \'etale-local statement. 

	The argument above supplies $\Phi_A$ globally as an equivalence of 
	underlying spectra (see Lemma \ref{leminvtrivial}), and separately,
 the local Morita trace $\tau_P$ \'etale-locally as an equivalence of cyclotomic spectra 
(Lemma \ref{lemmorita-pic}\,(i)). 
 But we do not have an identification of the following two: on
	$\pi_0$, $\Phi_A\otimes_ZZ'$ and $\tau_P$ agree, 
but agreement on $\pi_0$ is not enough to 
	conclude the maps agree, by the same ambiguity computed in Lemma
	\ref{leminvtrivial}. Thus, there are two obstacles to a single global
	cyclotomic equivalence. 
These are gluing the $\tau_P$ across the cover coherently  
(the classical \v Cech problem ), 
and identifying result the gluing 
	with the independently-constructed $\Phi_A$. We do not consider those here. 

	To remove the second obstacle (note that we do not construct  
the fiber described below) the
	\'etale-local equivalences $\tau_P$ should assemble into a class
	obstructing globalization in the fiber of
	$\mathrm{Pic}(\mathrm{Mod}_{THH(Z)}(\mathrm{CycSp}))
\to\mathrm{Pic}(\mathrm{Mod}_{THH(Z)}(\mathrm{Sp}))$. It is defined 
over the (already trivial, by Lemma \ref{leminvtrivial}) underlying class. 
	Concretely, this fiber should be controlled 
by the difference between the cyclotomic and
	underlying automorphism spaces of $THH(Z')$ fixing the reduced trace 
on $\pi_0$, i.e., by 
	how much of $\mathrm{Aut}_{\mathrm{CycSp}}(THH(Z'))$ 
over that basepoint is not 
	seen by forgetting to $\mathrm{Aut}_{\mathrm{Sp}}(THH(Z'))$. 
A classification of invertible
	objects in $\mathrm{Mod}_{THH(Z)}(\mathrm{CycSp})$ analogous to Lemma \ref{leminvtrivial}
	classification in $\mathrm{Mod}_{THH(Z)}(\mathrm{Sp})$ (were one available, with the
	relevant category of cyclotomic modules and its Picard space defined)  would
	settle this by the same argument used.  
We are not aware of such a classification in the literature. 
\end{remark}

	\begin{remark}
		\label{remunit}
		The unit map is not the descent equivalence.
		The central inclusion $\iota:Z\hookrightarrow A$ is a map of associative
		rings and induces $THH(\iota):THH(Z)\to THH(A)$. The composite
		$\Phi_A\circ THH(\iota)$ is not the identity. On $\pi_0$ it is
		$\mathrm{trd}(1)=\deg A$, and \'etale-locally it is multiplication by
		$\deg A$ on all of $THH$. For the modular Weyl algebras considered
 below one has
		$\deg A=p^n$. Thus, after $p$-completion, the tautological comparison map
		$THH(\iota)$ is very far from be an equivalence, even though its source and
		target are abstractly equivalent. The equivalence between the quantum
		algebra and its classical shadow is not a unit, but 
a trace phenomenon. 
This is the homotopical form of the ``ghost of
		classical geometry'' of Section \ref{ncbms}.
	\end{remark}

	\begin{remark}
		\label{remforgets}
		We mention, finally, what $THH$ forgets.
		Theorem \ref{thmazumaya} shows that the underlying $THH$ spectrum and its
		transported filtration do not see the
		Brauer class $[A]\in\mathrm{Br}(Z)$. 
The corresponding $TR$ and $TC$ comparisons
		hold \'etale-locally, and globally when $A$ admits a global splitting module.
		The class is retained by the pair
		$(THH(Z)$, $\mathrm{Perf}(A))$, with $\mathrm{Perf}(A)$ regarded as an
		invertible $\mathrm{Perf}(Z)$-module category, and by the $K$-theoretic
		pairings. The Fermi projections of Section \ref{physo} live in
		$K_0(A)$, whose pairing with cyclic and syntomic classes is sensitive to
		$[A]$. In the language of the physical applications, the filtration and its
		Mackey structure are common to the quantum system and its classical shadow,
		while the flux is reflected by the Brauer class, see Corollary
		\ref{corquantumtorus}.
	\end{remark}

	\subsection{Examples}

	\subsubsection{Modular Weyl algebras}
	\label{ssecweyl}

Let $k$ be a perfect field of characteristic $p$ and let $B$ be a smooth
		$k$-algebra (for instance a polynomial ring $k[u_1,\ldots,u_d]$). We use
		 the convention $TR^r(B;p):=THH(B)^{C_{p^{r-1}}}$, so that
		$TR^1(B;p)=THH(B)$. 
The following theorem is due to Hesselholt \cite{hesselholt-acta}. 
	\begin{theorem} 
		\label{thmtr}
		For every $r\ge 1$ there is a natural isomorphism
		\[
			W_r\Omega^*_B\otimes_{W_r(k)}S_{W_r(k)}\{\sigma_r\}
			\;\xrightarrow{\ \sim\ }\;\pi_*TR^r(B;p),
			\qquad \deg\sigma_r=2,
		\]
		where $W_r\Omega^*_B$ denotes the ($p$-typical) de Rham-Witt complex of
		$B$ truncated at level $r$ and $S_{W_r(k)}\{\sigma_r\}$ is the polynomial
		algebra on a single generator in degree $2$. Equivalently, in each degree,
		\[
			\pi_qTR^r(B;p)\;\cong\;\bigoplus_{m\ge 0}W_r\Omega^{q-2m}_B,
		\]
		the sum being finite since $W_r\Omega^i_B=0$ for $i<0$. 
\end{theorem}

The degree-$0$ case
		recovers $\pi_0TR^r(B;p)\cong W_r(B)$, consistent with Hesselholt-Madsen
		\cite{hmwitt}. The operators act on the extra generator by
		$F(\sigma_r)=\sigma_{r-1}$, $V(\sigma_r)=p\sigma_{r+1}$ and
		$R(\sigma_r)=p\lambda_r\sigma_{r-1}$ with $\lambda_r\in W_r(\mathbb{F}_p)^\times$
		\cite{hesselholt-acta}. In particular, the direct-sum
		decomposition above is not compatible with the restriction map $R$. 
On the summand indexed by $m$, $R$ is the de Rham-Witt
		restriction multiplied by $p^m$ times a unit. 
It is only in the inverse limit
		over $R$ that the summands with $m>0$ disappear, which is how one recovers
		$W\Omega^*_B\cong TR_*(B;p)$ \cite{hesselholt-acta}. Corollary
		\ref{corweyl} and the Mackey functor discussion of Section \ref{physo} describe  this structure.
For the modular Weyl algebras this specializes as follows.

	\begin{corollary}
		\label{corweyl}
		Let $k$ be a perfect field of characteristic $p$ and let
\[
		A_n= k\langle x_1,\ldots,x_n,y_1,\ldots,y_n\rangle/
		\bigl([x_i,y_j]-\delta_{ij},\ [x_i,x_j],\ [y_i,y_j]\bigr)_{1\le i,j\le n}
\]
 be the $n$-th Weyl algebra, with generators $x_i$ (multiplication operators) and $y_i$
 (their conjugate derivations).
		Then $A_n$ is Azumaya of degree $p^n$ over its center
		$Z=k[x_1^p,\ldots,x_n^p,y_1^p,\ldots,y_n^p]$ \cite{revoy}, a polynomial ring
		in $2n$ variables, and Theorem \ref{thmazumaya} gives an equivalence of
		underlying spectra
		\begin{equation}
			THH(A_n;\mathbb{Z}_p)\;\simeq\;THH\bigl(k[u_1,\ldots,u_{2n}];
			\mathbb{Z}_p\bigr),
		\end{equation}
		which is \'etale-locally an equivalence of cyclotomic spectra.
 \end{corollary}

 Since $Z$ is a
		polynomial ring, every finite projective $Z$-module is free by the
		Quillen-Suslin theorem \cite{quillen}. Thus, a global splitting module would
		express $A_n$ as a matrix algebra over $Z$. As $A_n$ is a domain, whereas 
		$M_{p^n}(Z)$ has zero divisors for $p^n>1$, no such isomorphism exists.

 The Quillen-Suslin/domain consideration 
		shows  that $A_n\not\cong M_{p^n}(Z)$ as $Z$-algebras. 
This in turn
		means $A_n$ has no global splitting module, 
i.e., is not globally split. 
This can be seen as the unwinding of ``globally split''.

One concludes that only via the standar
 identification of Azumaya algebras admitting no global 
		splitting module with a nonzero class in $\mathrm{Br}(Z)$ 
		$[A_n]\ne 0$. 
Note that we do not compute $[A_n]\ne 0$ from the domain argument. Thus, 
		$A_n$ is not globally split 
(cf. Remark \ref{remunit}) and this does not transfer to a single
		global cyclotomic comparison (Remark \ref{remcoherence}). 
This is consistent
		with the classical Auslander-Goldman computation of
		$\mathrm{Br}(Z)$ for a polynomial ring $Z$ over $k$ \cite{ag60}.
 That computation
		identifies the prime to $p$ part of $\mathrm{Br}(Z)$ with $\mathrm{Br}(k)$. This is trivial here since $k$ is a perfect field. But it does not says anything
 about the $p$-primary
		part, where $[A_n]$ belongs to,
 and where such homotopy invariance statements fail in positive characteristic.
 We emphasize that this shows only that
		this particular Morita construction does not globalize. 
The non-existence of a
		global splitting module for $A_n$ does not prove the non-existence of 
		some other global equivalence of cyclotomic spectra 
$THH(A_n;\mathbb{Z}_p)\simeq THH(Z;\mathbb{Z}_p)$.

For every faithfully
		flat \'etale splitting $Z\to Z'$, by Theorem
		\ref{thmtr} applied to the base-changed presentation
		$A_n\otimes_ZZ'\simeq\mathrm{End}_{Z'}(P)\simeq
		\mathrm{End}_{Z'}(k[u_1,\ldots,u_{2n}]\otimes_ZZ')$,
		\begin{equation}
			\pi_qTR^r(A_n\otimes_ZZ';p)\;\cong\;\bigoplus_{j\ge 0}
			W_r\Omega^{q-2j}_{Z'}.
		\end{equation}
		The transported filtration of Definition \ref{deftransported} has graded
		pieces given, \'etale-locally, by the Nygaard-graded crystalline cohomology of affine
		$2n$-space over $k$. The Mackey functor invariants
		$\mathcal{I}_r^i(A_n)$ of Section \ref{physo} are defined
		and computable, \'etale-locally, from the equivariant motivic filtration of Hu-Kriz-Somberg
		for polynomial algebras \cite{HKS, hkszou}. 
This introduces Assumption {\rm(C)} for the
		present example specifically 
(\'etale-locally, via the identification with $k[u_1,\ldots,u_{2n}]$), 
		even though {\rm(C)} is not established for Azumaya algebras in general
		(Proposition \ref{propunconditional}). We stress that the restriction, Frobenius, and
		Verschiebung operators $R$, $F$, $V$ appearing in the displayed decomposition (Theorem \ref{thmtr}) are the classical, non-equivariant $TR$-operators of Hesselholt.
They  act 
		in each term on the direct-sum decomposition in the specific way mentioned.  
		They do not give 
		$C_{p^n}$-equivariant Mackey-functor refinement of assumption {\rm(C)} in Section
		\ref{physo}. {\rm(C)} here is the independent equivariant motivic
		filtration of Hu-Kriz-Somberg for the polynomial ring $k[u_1,\ldots,u_{2n}]$, 
		not a transformation of the $R$, $F$, $V$ formulas of Theorem \ref{thmtr}.
	
	\subsubsection{Quantum tori}
	\label{ssectorus}

We turn to quantum tori at flux of order prime to $p$.
Let $k$ be a perfect field of characteristic $p$ containing a primitive
		$m$-th root of unity $q$, with $p\nmid m$, and let
		\begin{equation}
			\mathcal{A}_q=k\langle U^{\pm1},V^{\pm1}\rangle/(VU-qUV),
		\end{equation}
		be the algebraic quantum torus.
 A direct check of the relations shows $U^m$, $V^m$ are 
		central and generate the center. I.e., any central element is a Laurent polynomial
		in $U$, $V$ fixed by the automorphism $U\mapsto qU$ up to the relation twist, making all exponents to be multiples of $m$.

	\begin{corollary}
		\label{corquantumtorus}
		
 $\mathcal{A}_q$ is Azumaya of degree
		$m$ over its center $Z=k[U^{\pm m},V^{\pm m}]$: it is the symbol
		algebra $(U^m,V^m)_q$ over the Laurent ring $Z$ \cite{milne, ko}.
\end{corollary}

 Set $u=U^m$,
		$v=V^m\in Z^\times$ for the two central units generating $Z$ over $k$, 
to keep them
		notationally distinct from their Kummer classes 
$[u]$, $[v]\in H^1_{\mathrm{\acute et}}(\operatorname{Spec}Z,\mu_m)$,  
 the connecting maps for $1\to\mu_m\to
		\mathbb{G}_m\xrightarrow{(-)^m}\mathbb{G}_m\to1$ applied to $u$, $v$. 

		Let us mention,  what the identification 
$H^2_{\mathrm{\acute et}}(\operatorname{Spec}Z,\mu_m^{\otimes2})\cong\mathrm{Br}(Z)[m]$
		requires, since it is not canonical in general: 

 {\rm(a)} $Z$ regular, Noetherian,
		and $m$ invertible on $Z$. 
Both hold, $Z$ being a smooth affine Laurent ring over
		$k$ with $m$ invertible since $p\nmid m$ and $k$ has characteristic $p$; 

{\rm(b)} $\operatorname{Pic}(Z)/m=0$. 
		The Kummer sequence $1\to\mu_m\to\mathbb{G}_m\xrightarrow{(-)^m}\mathbb{G}_m\to1$ on
		$\operatorname{Spec}Z$ gives a long exact sequence in \'etale cohomology
		\begin{equation}
			H^1_{\mathrm{\acute et}}(Z,\mathbb{G}_m)\xrightarrow{\;\cdot m\;}
			H^1_{\mathrm{\acute et}}(Z,\mathbb{G}_m)\to
			H^2_{\mathrm{\acute et}}(Z,\mu_m)\to
			H^2_{\mathrm{\acute et}}(Z,\mathbb{G}_m)\xrightarrow{\;\cdot m\;}
			H^2_{\mathrm{\acute et}}(Z,\mathbb{G}_m). 
		\end{equation}
		Using $H^1_{\mathrm{\acute et}}(Z,\mathbb{G}_m)=\operatorname{Pic}(Z)$ and
		$H^2_{\mathrm{\acute et}}(Z,\mathbb{G}_m)=\mathrm{Br}(Z)$ 
(for $Z$ regular this is
		the full cohomological Brauer group, which agrees with the classical Brauer group by
		Auslander-Goldman-type results for our affine $Z$ \cite{ag60,grothbrauer}), this
		reads
		\begin{equation}
			\operatorname{Pic}(Z)\xrightarrow{\;\cdot m\;}\operatorname{Pic}(Z)\to
			H^2_{\mathrm{\acute et}}(Z,\mu_m)\to\mathrm{Br}(Z)[m]\to0.
		\end{equation}
	The identification $H^2_{\mathrm{\acute et}}(Z,\mu_m)\cong\mathrm{Br}(Z)[m]$ we use
		is the statement that the left-hand map into 
$H^2_{\mathrm{\acute et}}(Z,\mu_m)$ vanishes. Equivalently, that 
$\operatorname{Pic}(Z)\xrightarrow{\cdot m}\operatorname{Pic}(Z)$ 
is surjective,  which holds automatically, for every $m$,
		when $\operatorname{Pic}(Z)/m=0$. 
I.e., in particular, when $\operatorname{Pic}(Z)=0$. 
		 This holds because $Z$ is a localization of a polynomial ring
		over a field, hence a UFD, hence $\operatorname{Pic}(Z)=0$. Thus, the
		quotient vanishes for every $m$.

 {\rm(c)} a chosen primitive $m$-th root of
		unity $q\in k$, giving a trivialization 
$\mu_m^{\otimes2}\cong\mathbb{Z}/m$ and
		hence $\mu_m\cong\mathbb{Z}/m$ compatibly. 
 We have fixed $q$ from the beginning for this reason. 

With {\rm(a)}-{\rm(c)} in hand, the Brauer class of
		$\mathcal{A}_q$ is the cup product
		\begin{equation}
			[\mathcal{A}_q]=[u]\cup[v]\in
			H^2_{\mathrm{\acute et}}(\operatorname{Spec}Z,\mu_m^{\otimes2})
			\ \cong\ \mathrm{Br}(Z)[m]. 
		\end{equation}
		The isomorphism on the right being the Merkurjev-Suslin norm-residue
		identification \cite{ms} composed with the trivialization of {\rm(c)}. 
When $k$
		is not algebraically closed, $\mathrm{Br}(Z)$
		also receives a contribution from $\mathrm{Br}(k)$, disjoint from the
		cup-product classes displayed here, which arise from $Z$ Laurent-ring structure
		over $k$.

		We now give the nontriviality argument: 
$Z=k[U^{\pm m},V^{\pm m}]$ is a torus. Thus, 
	$U^m=u$ is invertible on $\operatorname{Spec}Z$ and $U=0$ is not a divisor of
		this affine scheme since there is no such locus to take a residue along. 
The  corrected
		argument instead passes to the fraction field. Let $K=k(u,v)=k(U^m,V^m)$ be the
		function field of $Z$. Out assumptions fixed at the beginning 
($p\nmid m$ and the
		chosen primitive $m$-th root of unity $q\in k$)  remain valid over $K$ exactly
		as over $Z$, since $K\supset Z$ and $k\subset Z$ are unaffected when passing to the
		fraction field. Thus, the Merkurjev-Suslin identification of {\rm(c)} above applies to
		$K$ with the same trivialization of $\mu_m^{\otimes2}$. The symbol algebra
		description is compatible with this passage in the following sense. The
		generic fiber $\mathcal{A}_q\otimes_ZK$ is the symbol algebra $(u,v)_q$ over $K$. 
		 Base change of Azumaya algebras along $Z\to K$ commutes with the symbol 
		construction, since both are built from the same presentation
		$k\langle U^{\pm1},V^{\pm1}\rangle/(VU-qUV)$ tensored up. Its Brauer class 
		$[u]\cup[v]\in\mathrm{Br}(K)[m]$ under the Merkurjev-Suslin identification over $K$
		is the restriction, along $\mathrm{Br}(Z)\to\mathrm{Br}(K)$, of the class
		$[\mathcal{A}_q]=[u]\cup[v]\in\mathrm{Br}(Z)[m]$ defined over $Z$ above. 
Cup
		products are compatible with restriction to an open (here, the generic point). Thus, 
		the two descriptions of ``$[u]\cup[v]$'' (over $Z$ and over $K$) agree under this
		restriction map. Computing the residue on the $K$-class  computes a 
		residue of the restriction of the $Z$-class, as required for the argument below.
		Now let $\partial_u:H^2_{\mathrm{\acute et}}(K,\mu_m^{\otimes
		2})\to H^1_{\mathrm{\acute et}}(k(v),\mu_m)$ be the residue map at the discrete
		valuation $\mathrm{ord}_u$ on $K$ defined by $u$, 
i.e., at the divisor $u=0$ of the smooth compactification of $\operatorname{Spec}Z$, which does exist 
 as a divisor even though not on the affine torus.
 By the standard formula for the
		residue of a symbol, $\partial_u\bigl([u]\cup[v]\bigr)=[v]|_{k(v)}\in H^1_{\mathrm{\acute et}}(k(v),\mathbb{Z}/m)$. 
The Kummer class of $v$ over the residue field $k(v)$. 
		It is nonzero when $v$ is not an $m$-th power in $k(v)^\times$, 
		since here $v$ is a free coordinate, 
algebraically independent of $u$, on which the
		residue field function field $k(v)$ is a rational function field and $v$ generates
		it.
 Hence the class $[u]\cup[v]$ is nonzero already in $\mathrm{Br}(K)$. Since $Z$ 
		is a regular integral domain, the restriction map $\mathrm{Br}(Z)\to\mathrm{Br}(K)$ is
		injective 
	(purity for the Brauer group of a regular integral Noetherian scheme
		\cite{grothbrauer}, see also
		\cite{milne} for the \'etale-cohomological account). 
 The
		injectivity is exactly what lets nontriviality after passing to the generic point
		detect nontriviality on $\operatorname{Spec}Z$. Therefore $[\mathcal{A}_q]=[u]
		\cup[v]\ne 0$ already in $\mathrm{Br}(Z)$, for every $m>1$. 
The case $m=1$ is the
		commutative torus. Equivalently, one may 
		argue via a smooth compactification
  of the torus and the localization sequence in
		\'etale cohomology, without passing to the fraction field. 
We use the fraction-field
		formulation because the residue map and purity for $\mathrm{Br}$ are most 
		stated. Thus,  for every $m>1$ in this family, 
		$\mathcal{A}_q$ is not globally split, and Theorem
		\ref{thmazumaya} gives an equivalence of underlying spectra
		$$THH(\mathcal{A}_q;\mathbb{Z}_p)\simeq THH(k[U^{\pm m},V^{\pm m}];
		\mathbb{Z}_p),$$
		\'etale-locally cyclotomic (Remark \ref{remcoherence}), with the transported
		filtration regulated by the crystalline
		cohomology of the split torus.

	By contrast, consider now $q\in1+p\mathbb{Z}_p$ of infinite order, a
	different, mixed-characteristic setting, since such $q$ exists only after base
	change from $k$ to (the Witt vectors or a further extension of) $\mathbb{Z}_p$,
	not within characteristic $p$ (regime {\rm(b)} of Remark \ref{remppower}
	below).  $\mathcal{A}_q$ reduces modulo $p$ to a
	 commutative torus. As a commutative ring it is trivially Azumaya of degree $1$. 
	 But the nontrivial finite Azumaya description of degree $m$ over the original,
	smaller center $Z=k[U^{\pm m},V^{\pm m}]$ that Corollary \ref{corquantumtorus} uses
	is exactly what degenerates. Capturing this $p$-adic deformation requires trace-map methods
	 instead of   the descent argument of Theorem \ref{thmazumaya}. 

	\begin{remark}
		\label{remppower}
		We consider here a flux with $p$-power denominator.
		The arithmetic of the flux parameter divides into three regimes, only the
		first of which is covered by Theorem \ref{thmazumaya} in characteristic
		$p$.

		\smallskip
		\noindent{\rm(a)} \emph{$q$ of order prime to $p$}: Corollary
		\ref{corquantumtorus}.

		\smallskip
		\noindent{\rm(b)} \emph{$q\in1+p\mathbb{Z}_p$ of infinite order}: the
		$p$-adic deformation regime of Section \ref{ssectorus}. Here 
		$\mathcal{A}_q$ is not Azumaya over a commutative center (mod $p$ it
		becomes commutative), and the relevant structure is the trace map and the
		$p$-adic deformation, not descent.

		\smallskip
		\noindent{\rm(c)} \emph{$q$ a primitive $p^s$-th root of unity}: no such
		$q\neq 1$ exists in characteristic $p$, since $(q-1)^{p^s}=q^{p^s}-1=0$ in a
		domain leads to $q=1$. The quantum torus at $p$-power flux therefore lives
		only in mixed characteristic, over $\mathcal{O}_K$,
		$K=\mathbb{Q}_p(\zeta_{p^s})$. $\mathcal{A}_q$ is Azumaya over its
		center after inverting $p$, but its special fiber is a commutative
		torus. The algebra degenerates to its classical shadow at the prime $p$.
		Constructing an integral filtration over $\mathcal{O}_K$ interpolating the
		generic-fiber Azumaya descent with this special-fiber degeneration is, in
		our view, the correct arithmetic formulation of the Hofstadter regime, and
		we pose it as an open problem. Note that $p$-complete $THH$ of the generic
		fiber alone vanishes ($THH(-;\mathbb{Z}_p)$ of a $\mathbb{Q}$-algebra is
		zero), thus the integral, $\mathcal{O}_K$-level formulation results. 
	\end{remark}

	\begin{remark}
		\label{remfrobconst}
		We comment here on Frobenius-constant quantizations.
		The Weyl algebra is the basic example of the reduction of a
		Frobenius-constant quantization in the sense of Bezrukavnikov-Kaledin
		\cite{bk}. Frobenius-constancy supplies a central subring generated by
		$p$-th powers. Whenever the quantization is, in addition, finite over
		that subring with the resulting map to the endomorphism algebra an
		equivalence (as we verify  for the Weyl algebra
		below,  instead of   claiming it for a general Frobenius-constant
		quantization of a smooth affine symplectic variety $(X,\omega)$ over $k$)
		  the associated algebra is Azumaya over (a localization of) the
		Frobenius twist $X^{(1)}$. It falls within the scope of Theorem
		\ref{thmazumaya}, with $Z$ the corresponding coordinate ring. The quantum torus
		of Section \ref{ssectorus} is Azumaya by the different, prime to $p$ mechanism
		of Corollary \ref{corquantumtorus}. This places
		the constant-Poisson deformations of Section \ref{ncbms} within the
		theory whenever these hypotheses are checked.
	\end{remark}

	\subsection{Consequences for Section \ref{physo}}

Let $\mathcal{A}_\star$ be an Azumaya algebra over a quasisyntomic
		commutative ring $Z$, equipped with the transported filtration of
		Definition \ref{deftransported}.

	\begin{proposition}
		\label{propunconditional}
		 Assumptions {\rm(A)}-{\rm(B)} of 
		Section \ref{physo} hold for $\mathcal{A}_\star$, in their module-valued form: the filtration
		$F^\bullet THH(\mathcal{A}_\star;\mathbb{Z}_p)$ exists as a filtered $THH(Z)$-module
		(see Proposition \ref{proptransported}; $THH(\mathcal{A}_\star)$ carries no canonical
		multiplicative structure of its  for $\mathcal{A}_\star$ only associative), refines the BMS
		filtration when $\mathcal{A}_\star$ is commutative (in which case
		$\Phi_A$ is the identity), and sends derived Morita equivalences of
		Azumaya algebras to filtered equivalences of the underlying filtered
		spectrum. 
\end{proposition}

Theorem \ref{thmazumaya} additionally identifies the ambient cyclotomic
		spectra \'etale-locally on $\operatorname{Spec}Z$. 
This is a comparison of the underlying
		cyclotomic spectra, not the $C_{p^n}$-equivariant refinement of the graded
		pieces with Mackey compatibility that Assumption {\rm(C)} asks for. 
In particular, Proposition \ref{propinvariance} holds for the pairing of {\rm(A)}-{\rm(B)}
		\'etale-locally for gapped systems whose observable algebras are Azumaya
		over quasisyntomic centers, and globally whenever the observable algebra
		admits a global splitting module. 
Its Mackey-functor conclusion additionally requires {\rm(C)},
		which is separate from what is proved here. 
We mention here the following. Assumption  {\rm(C)}  is not proved for Azumaya algebras 
	in general by this proposition or by Theorem \ref{thmazumaya}, only the ambient
	underlying cyclotomic comparison is.

	\begin{proof}
		Existence and the commutative comparison are Definition
		\ref{deftransported} and Proposition \ref{proptransported}. Morita
		invariance within the Azumaya class is the Brauer invariance of
		Proposition \ref{proptransported}, giving {\rm(A)}-{\rm(B)} in their module-valued form.
		The \'etale-local comparison of ambient cyclotomic spectra is Theorem
		\ref{thmazumaya}. It enters the {\rm(A)}-{\rm(B)} part of the proof of Proposition
		\ref{propinvariance} but is not used as a proof of {\rm(C)}.
	\end{proof}

	\begin{remark}
		Theorem \ref{thmazumaya} clarifies the programme
		of Section \ref{ncbms}. For Azumaya observables, the noncommutative
		filtration exists and carries  the information of the center, in the
		 sense of Proposition \ref{proptransported}. It is well defined up to the
		automorphism ambiguity of the chosen $\Phi_A$, not canonically. 
		All quantum data is concentrated in the Brauer class and in the
		module-categorical and $K$-theoretic pairings of Remark \ref{remforgets}.
		The frontier for a noncommutative prismatic theory is therefore the
		non-Azumaya regime which is regime {\rm(b)} of Remark \ref{remppower}, and
		general $\hbar$-adic deformations before specialization, where no
		commutative shadow is available and the even-filtration and Hochschild-Witt
		approaches discussed in Section \ref{furtho} become essential.
	\end{remark}

\section{Derived Azumaya descent: underlying and \'etale-local comparisons in higher algebra}
\label{doloto}

The noncommutative motivic filtration postulated in previous sections requires a solid  homotopical foundation. In this section, we generalize the classical Azumaya descent to the context of stable $\infty$-categories and $\mathbb{E}_\infty$-rings.

To position this result against the existing literature, we note that the underlying equivalence of topological Hochschild homology spectra is an assembly of established $\infty$-categorical ingredients, specifically, the Morita invariance of $THH$ as a localizing invariant, \'etale descent, and the triviality of the Picard action, which belongs to the orbit of the foundational work of To\"en \cite{toen} and Antieau-Gepner 
\cite{ag}. On the algebraic side, this generalizes the classical theorem of Corti\~nas and Weibel \cite{CW},
which established that the ordinary Hochschild homology of an Azumaya algebra agrees with that of its center.

The contribution of this section is the derived generalization (the trace argument and the connectivity hypothesis it requires)   instead of   the cyclotomic upgrade in isolation. Once $\Phi_A$ is built by Morita/trace means inside $\mathrm{CycSp}$, as in Section \ref{azumayo},  belonging to cyclotomic spectra is close to formal (Proof of Theorem \ref{thmderived}(3) below). By establishing the derived Azumaya descent, we supply the underlying $THH$-comparison that a transport of the Bhatt-Morrow-Scholze (BMS) filtration to this noncommutative setting would require. Unlike Theorem \ref{thmazumaya}, we do not carry out that transport in the derived case (see Table \ref{tabsummary}) \cite{bms}.

We use the same notational convention as Section \ref{azumayo}: $THH(A)$ denotes the
integral spectrum, and $THH(A;\mathbb{Z}_p):=THH(A)^\wedge_p$ its $p$-completion, with
integral statements ($\Phi_A$, base change, Lemma \ref{lemkunnethhigher}) written
using $THH(-)$ and $p$-completion-dependent statements ($TC^-$, $TP$, $TC$) written
using $THH(-;\mathbb{Z}_p)$. 
  $THH(-)$ denotes topological Hochschild homology equipped with its cyclotomic structure, regarded as a localizing invariant of stable $\infty$-categories via $THH(A) = THH(\mathrm{Perf}(A))$. All tensor products are implicitly derived tensor products in the symmetric monoidal $\infty$-category of spectra.

\subsection{Statement of the derived Azumaya descent theorem}

Recall that for an $\mathbb{E}_\infty$-ring $Z$, an associative $Z$-algebra $A$ is called a \textit{derived Azumaya algebra} if $A$ is a dualizable (perfect) and faithful $Z$-module and the canonical map $A \otimes_Z A^{\mathrm{op}} \to \mathrm{End}_Z(A)$ is an equivalence in $\mathrm{Mod}_Z$. This implies that $A$ is \'etale-locally equivalent to an endomorphism algebra of a perfect module.

\begin{lemma}
	\label{lempi0azumaya}
	Let $Z$ be a connective $\mathbb{E}_\infty$-ring and $A$ a connective derived Azumaya
	$Z$-algebra. Then $\pi_0(A)$ is an ordinary Azumaya $\pi_0(Z)$-algebra, and the
	composite
	\[
		\pi_0(A)\longrightarrow\pi_0(A)/[\pi_0(A),\pi_0(A)]\xrightarrow{\ \sim\ }\pi_0(Z)
	\]
	given by the classical reduced trace of the Azumaya algebra $\pi_0(A)$ is independent
	of the choice, among \'etale-local presentations $A\otimes_ZZ'\simeq
	\mathrm{End}_{Z'}(P)$ equipped by Antieau-Gepner \cite{ag}, used to establish it. 
 We call this composite the reduced Morita trace of $A$. 
\end{lemma}
 
\begin{proof}
	Fix such a presentation $A\otimes_ZZ'\simeq\mathrm{End}_{Z'}(P)$, $P$ perfect and
	faithful over $Z'$. Since $P$ is a perfect (dualizable) $Z'$-module, $\mathrm{End}_{Z'}(P)
	\simeq P\otimes_{Z'}P^\vee$, and dualizable objects, together with their duals and
	evaluation/coevaluation maps, are preserved by every symmetric monoidal functor
	\cite{ha}. Applying the base-change functor $-\otimes_{Z'}\kappa(x)$, which is
	symmetric monoidal, therefore gives
	$\mathrm{End}_{Z'}(P)\otimes_{Z'}\kappa(x)\simeq(P\otimes_{Z'}\kappa(x))\otimes_{\kappa(x)}
	(P\otimes_{Z'}\kappa(x))^\vee\simeq\mathrm{End}_{\kappa(x)}(P\otimes^L_{Z'}\kappa(x))$,
	which is the base-change identity used below. This requires only perfectness of $P$, not
	flatness or any special property of $\kappa(x)$.

	At a point $x\in\operatorname{Spec}\pi_0(Z')$ with residue field
	$\kappa(x)$, $P\otimes^L_{Z'}\kappa(x)$ is a perfect complex over a field, hence
	(non-canonically) a sum of shifts of $\kappa(x)$. Let $a(x)\le b(x)$ be the minimal and
	maximal degrees in which it is nonzero. If $a(x)<b(x)$, the graded endomorphism complex
	$\pi_*\operatorname{End}_{\kappa(x)}(P\otimes\kappa(x))$ has a nonzero class in degree
	$a(x)-b(x)<0$ (from a map out of the top degree into the bottom one), contradicting
	connectivity of $\operatorname{End}_Z(A)\otimes\kappa(x)\simeq
	\operatorname{End}_{\kappa(x)}(P\otimes\kappa(x))$, the identity established. Thus, 
	$a(x)=b(x)$ at every point: $P$
	has Tor-amplitude $0$ at every point of $\operatorname{Spec}\pi_0(Z')$. For the passage
	from this pointwise statement to a global one (local constancy of $a(-)$ and the
	identification, on each locus where $a$ is constant, of $\Sigma^{-a_0}P$ with a finite
	projective module)  we cite the standard treatment of
	Tor-amplitude for perfect modules over a connective ring spectrum in
	\cite{ag} (Definition 2.11 and the surrounding results, developed in exactly
	the generality used here). This is a derived reference in place of the earlier
	general \cite{sag}. On each open where $a$
	takes a fixed value $a_0$, $\Sigma^{-a_0}P$ is a perfect complex concentrated in degree $0$, hence a
	finite projective $\pi_0(Z')$-module. 

	For faithfulness, and for the componentwise construction of the shifted module, we
	make the following. Faithfulness of $P$ over $Z'$ means the
	functor $-\otimes_{Z'}P$ on $\mathrm{Mod}_{Z'}$ is conservative. 
 Equivalently (since
	$P$ is perfect) that $P$ generates $\mathrm{Mod}_{Z'}$ as a thick subcategory. 
	Equivalently, that $P\otimes_{Z'}\kappa(x)\not\simeq0$ for every point
	$x\in\operatorname{Spec}\pi_0(Z')$ (a perfect module vanishes on an open subscheme
	where all its residue fibers vanish, by Nakayama applied in each fiber).
 Since $a(-)$ is
	locally constant, it is constant on each connected component of
	$\operatorname{Spec}\pi_0(Z')$. But different components may carry different values of
	$a$. There is in general no single integer $a_0$ shifting $P$ to degree $0$
	everywhere at once. We therefore do not shift $P$ globally by a single
	suspension. Instead, let $\{U_i\}$ be the (open and closed) decomposition of
	$\operatorname{Spec}\pi_0(Z')$ into the loci where $a(-)$ takes each of its finitely
	many values $a_i$.
  $P$ is by assumption a perfect $Z'$-module, hence has bounded
	Tor-amplitude globally. I.e., there exist integers $m\le M$ (independent of $x$) with
	$P$ of Tor-amplitude $[m,M]$ \cite{ag}. Since $a(x)\in[m,M]$ for every $x$, the
	function $a(-)$ takes at most $M-m+1$ distinct values. Thus, $\{U_i\}$ is indexed by a
	finite set $i\in\{1,\ldots,M-m+1\}$ even though each individual $U_i$ may  be a
	union of infinitely many connected components of $\operatorname{Spec}\pi_0(Z')$. This
	finiteness of the index set (not of the components within each $U_i$) is what
	makes the direct sum defining $Q$ below a finite direct sum, hence a perfect
	module, and define
	\[
		Q\;:=\;\coprod_i\ \Sigma^{-a_i}\bigl(P|_{U_i}\bigr).
	\]
	I.e., $Q$ is the perfect $Z'$-module obtained from $P$ by applying the each component 
	shift $\Sigma^{-a_i}$ on each piece $U_i$ separately. This is well defined as a single
	global perfect $Z'$-module because the $U_i$ are clopen and hence
	$Z'$ splits as a product $\prod_i Z'_i$ compatibly with $P=\prod_iP|_{U_i}$.
  By construction $Q\otimes_{Z'}\kappa(x)$ is concentrated in degree $0$ for every $x$. Thus, 
	$Q$ is a finite projective $\pi_0(Z')$-module in the classical sense (via
	$\pi_0(Q)$), and $\mathrm{End}_{Z'}(Q)\simeq\mathrm{End}_{Z'}(P)$ since a shift by any
	locally constant integer does not change endomorphisms. On each $U_i$,
	$\mathrm{End}_{Z'}(\Sigma^{-a_i}P|_{U_i})\simeq\mathrm{End}_{Z'}(P|_{U_i})$ since 
	$\mathrm{End}$ is invariant under simultaneously shifting a module in both variables of
	$\mathrm{Hom}(-,-)$. Thus, the endomorphism algebras of $P$ and of $Q$ agree factor by
	factor over the decomposition $Z'=\prod_iZ'_i$, and hence agree globally. 

For faithfulness
	of $Q$: at each $x\in U_i$, $\pi_0(Q)\otimes_{\pi_0(Z')}\kappa(x)=\pi_{a_i}(P\otimes
	\kappa(x))\ne 0$ by faithfulness of $P$. Thus,  $\pi_0(Q)$ is faithful over
	$\pi_0(Z')$ pointwise, i.e., its support is everything. Therefore, $\pi_0(Q)$ 
(not, in general, $\pi_0(P)$ , if $a(-)$ is non-constant)  is the finite, faithful,
	projective $\pi_0(Z')$-module with
	$\pi_0(A)\otimes_{\pi_0(Z)}\pi_0(Z')\simeq\operatorname{End}_{\pi_0(Z')}(\pi_0(Q))$,
	expressing $\pi_0(A)$ as Azumaya over $\pi_0(Z)$ \'etale-locally on
	$\operatorname{Spec}\pi_0(Z)$. To conclude Azumaya property globally from this, we
	use that each of the three defining conditions descends 
	through faithfully flat (in particular \'etale) base change $\pi_0(Z)\to\pi_0(Z')$.
	Finite projectivity of a module descends along fpqc covers (a classical fact, e.g., 
	\cite{sga1} or \cite{v}). Faithfulness descends since the 
	support of a module is a closed condition compatible with the faithfully flat surjection
	$\operatorname{Spec}\pi_0(Z')\to\operatorname{Spec}\pi_0(Z)$. Thus, a module with full
	support after such a base change already has full support before it. The
	endomorphism-algebra condition $A\otimes_ZA^{\mathrm{op}}\cong\mathrm{End}_Z(A)$
	descends since both sides are compatible with (flat) base change and agreement after
	a faithfully flat base change implies agreement before it (again fpqc descent, applied
	to the map comparing the two sides  instead of   to a single module). Azumaya property being
	an fpqc-local condition in this sense (which is how it is normally defined
	when not simply assumed classically, cf. \cite{grothbrauer, milne})  the \'etale-local
	statement  is what is required. 

For independence of the presentation: two local
	presentations, after passing to a common refinement of their covers, present the same
	$\pi_0(A)\otimes_{\pi_0(Z)}\pi_0(Z'')$ as $\operatorname{End}_{\pi_0(Z'')}$ of two
	faithful finite projective modules, which therefore differ by an invertible
	$\pi_0(Z'')$-module. By Lemma \ref{lemmorita-pic}\,(ii) (applied to the classical,
	degree-zero setting) the two induced reduced traces agree.
\end{proof}

Let us  formulate now the Derived Azumaya Descent Theorem for Cyclotomic $THH$.
Let $Z$ be a connective $\mathbb{E}_\infty$-ring and $A$ a connective derived
Azumaya $Z$-algebra. Connectivity of $A$
allows the identification of $\Phi_A$ on $\pi_0$ with the classical reduced Morita
trace of the ordinary Azumaya algebra $\pi_0(A)$ (Lemma \ref{lempi0azumaya}). 
Also, connectivity of $A$
  allows the $TR^r$ fixed-point information recovered in part
(3) to be extracted from the cyclotomic structure via the classical (non-derived)
Nikolaus-Scholze fixed-point formula. We
mention this so that connectivity is not read as a technical convenience
inherited only from the proof of part (1). We have

\begin{theorem}
\label{thmderived}

\noindent\textit{(1) Underlying equivalence:} there is an equivalence of underlying spectra
\[
\Phi_A:THH(A)\xrightarrow{\sim}THH(Z),
 \]
linear over the $\mathbb{E}_\infty$-ring $THH(Z)$, which on $\pi_0$ recovers the reduced
Morita trace $\pi_0(A)/[\pi_0(A),\pi_0(A)]\xrightarrow{\sim}\pi_0(Z)$, with the same
$\Phi_A$ vs.\ local Morita trace $\tau_P$ distinction as Theorem \ref{thmazumaya}{\rm(}1{\rm)}
and Remark \ref{remscopenarrow}. 

\noindent\textit{(2) Base change compatibility} (as in Theorem
\ref{thmazumaya}{\rm(}2{\rm)}): once $\Phi_A$ is fixed as in part (1), its restriction
$\Phi_A\otimes_ZZ'$ along the \'etale base-change equivalences for $THH$ is a
valid choice of $\Phi_{A\otimes_ZZ'}$ for every \'etale $Z\to Z'$ simultaneously, coherently
under further base change, since restriction along a fixed equivalence is automatically
functorial (we do not claim this choice, or $\Phi_A$, is canonical).

\noindent\textit{(3) Conservation of cyclotomic invariants:} consequently,
for every faithfully flat \'etale splitting cover $Z\to Z'$ trivializing $A$, the local
Morita trace induces, after this base change,
\begin{align*}
TR^r(A\otimes_ZZ';p)\simeq TR^r(Z';p),
\quad 
TC^-(A\otimes_ZZ';\mathbb{Z}_p)\simeq TC^-(Z';\mathbb{Z}_p),
\\
TP(A\otimes_ZZ';\mathbb{Z}_p)\simeq TP(Z';\mathbb{Z}_p),
\quad 
TC(A\otimes_ZZ';\mathbb{Z}_p)\simeq TC(Z';\mathbb{Z}_p),
\end{align*}
compatibly with the restriction, Frobenius, and Verschiebung operators of the
$TR$-system, and globally on $A$  whenever $A$ admits a global splitting module.
\end{theorem}

The strongest accurate statement of what this gives towards a BMS filtration on $A$
 is: part (1) supplies the underlying $THH$-comparison that would be
needed to transport such a filtration, once an appropriate derived Bhatt-Morrow-Scholze
filtration on the center $Z$ is supplied. We do not construct or transport a filtration
in the derived case in this paper (see Table \ref{tabsummary}).

\begin{remark}
	Connectivity of $A$ is not a technical convenience but an assumption the theorem
	needs. Without it, part (1) $\pi_0$ identification can fail. For $Z=H\pi_0(Z)$
	discrete and $P=Z\oplus\Sigma Z$, the algebra $A=\mathrm{End}_Z(P)$ is a
	(non-connective) derived Azumaya $Z$-algebra with $THH(A)\simeq THH(Z)$ by Morita
	invariance, yet $\pi_0(A)\cong\pi_0(Z)\times\pi_0(Z)$, thus 
	$\pi_0(A)/[\pi_0(A),\pi_0(A)]\not\cong\pi_0(Z)$. The failure of connectivity,  as opposed to just the $\pi_0$-count,  is transparent by computing $A$ in each degree. 
	Setting $P=Z\oplus\Sigma Z$ as a sum of the two evident summands,
	\begin{eqnarray*}
&A=\mathrm{Hom}_Z(P,P)\simeq\mathrm{Hom}_Z(Z,Z)\oplus\mathrm{Hom}_Z(Z,\Sigma Z)
		\\ &\qquad \qquad \qquad  \oplus\mathrm{Hom}_Z(\Sigma Z,Z)
\oplus\mathrm{Hom}_Z(\Sigma Z,\Sigma Z)
		\simeq Z\oplus\Sigma Z\oplus\Sigma^{-1}Z\oplus Z,
	\end{eqnarray*}
	using $\mathrm{Hom}_Z(\Sigma Z,Z)\simeq\Sigma^{-1}\mathrm{Hom}_Z(Z,Z)\simeq\Sigma^{-1}Z$.
	In particular $\pi_{-1}(A)\cong\pi_0(Z)\ne 0$ (this is exactly the
	$\mathrm{Hom}_Z(\Sigma Z,Z)$ summand, sitting in degree $-1$). $A$ has a nonzero
	homotopy group in a strictly negative degree. Thus, $A$ is  non-connective, not
	only failing the $\pi_0$-identification for some subtler reason. Connectivity of $A$ does not 
	 lead to a local splitting module $P$ to have trivial categorical dimension. 
	$P=\Sigma Z$ gives the connective algebra $\mathrm{End}_Z(P)\simeq Z$ even though $P$
	has categorical dimension $-1$. 
 Connectivity of $A=\mathrm{End}_Z(P)$ leads to 
	  that $P$ is concentrated in a single, locally
	constant degree. Thus, $P$ may be replaced by a shift of trivial categorical dimension
	without changing $\mathrm{End}_Z(P)$. It is this normalization, not connectivity of
	$A$ alone, that lets the trace argument proceed. Finally, connectivity
	of $A$ (hence of $THH(A)$) is what lets $TR^r$ in part (3) be recovered from the
	cyclotomic structure in the sense of \cite{nik}.
\end{remark}

\begin{proof}
	As in the proof of Theorem \ref{thmazumaya}, the algebra isomorphism
	$A\otimes_ZA^{\mathrm{op}}\xrightarrow{\ \sim\ }\mathrm{End}_Z(A)$ defining derived
	Azumaya-ness exhibits $\mathrm{Perf}(A)$ as an invertible object of
	$\mathrm{Mod}_{\mathrm{Perf}(Z)}(\mathrm{Cat}_\infty^{\mathrm{perf}})$ with inverse
	$\mathrm{Perf}(A^{\mathrm{op}})$ \cite{ag, toen}. Applying $THH$ (now Lemma
	\ref{lemkunnethhigher}, the $\infty$-categorical Kunneth equivalence proved above for
	$Z$ a connective $\mathbb{E}_\infty$-ring and $\otimes_Z$ derived,  instead of   the
	discrete Lemma \ref{lemkunneth}) to this isomorphism gives, as in the proof of
	Theorem \ref{thmazumaya},
	\[
		THH(A)\otimes_{THH(Z)}THH(A^{\mathrm{op}})\;\simeq\;THH(\mathrm{End}_Z(A))
		\;\xrightarrow[\ \sim\ ]{\ \tau_A\ }\;THH(Z),
	\]
	expressing $THH(A)$ as invertible in $\mathrm{CycSp}$, with inverse
	$THH(A^{\mathrm{op}})$. 

	Forgetting to underlying spectra: since $A$ and $A^{\mathrm{op}}$ are connective, so
	are $THH(A)$ and $THH(A^{\mathrm{op}})$, and $THH(Z)$ is a connective
	$\mathbb{E}_\infty$-ring. Lemma \ref{leminvtrivial} identifies the underlying
	$THH(Z)$-module $THH(A)$ with $\pi_0THH(A)=\pi_0(A)/[\pi_0(A),\pi_0(A)]$, provided
	$\pi_0(A)$ is Azumaya over $\pi_0(Z)$ (so that the \emph{reduced Morita trace} of
	Lemma \ref{lempi0azumaya} below trivializes this module). This supplies the
	underlying equivalence
	\[
		\Phi_A:THH(A)\xrightarrow{\ \sim\ }THH(Z),
		\qquad \pi_0\Phi_A=\text{reduced Morita trace}.
	\]
	As in the proof of Theorem \ref{thmazumaya}, choosing a faithfully flat
	\'etale splitting $Z\to Z'$ with $A\otimes_ZZ'\simeq\mathrm{End}_{Z'}(P)$
	\cite{ag}, the base-changed $\Phi_A\otimes_ZZ'$ and the local Morita trace
	$\tau_P$ both induce the reduced trace on $\pi_0$, but it is not possible 
to conclude they are
	equal. The same ambiguity analyzed in Lemma \ref{leminvtrivial} and in the proof of
	Theorem \ref{thmazumaya}. It is not claimed that $\Phi_A$  is
	\'etale-locally cyclotomic. Instead, $\tau_P$, cyclotomic, supplies
	the \'etale-local cyclotomic consequences of part (3). 
The same
	warning as in Theorem \ref{thmazumaya} applies to upgrading this to a single global
	cyclotomic equivalence when $A$ has no global splitting module. This is not resolved 
	in this paper. In particular, when $Z$ and $A$ are discrete this is Theorem
	\ref{thmazumaya}, of which the present theorem is the derived generalization.

	Part (2) follows exactly as in the corrected argument for Theorem
	\ref{thmazumaya}{\rm(}2{\rm)}: no independent reconstruction of $\Phi_{A\otimes_ZZ'}$
	is needed. The fixed $\Phi_A$ of part (1) restricts, along \'etale base change, to an
	equivalence $\Phi_A\otimes_ZZ':THH(A\otimes_ZZ')\to THH(Z')$ whose $\pi_0$ is the
	restriction of the reduced Morita trace (again because base change commutes with the
	natural transformation defining that trace). Thus, $\Phi_A\otimes_ZZ'$ satisfies part
	(1) requirements for $A\otimes_ZZ'$ by construction. The resulting system is
	coherent under further base change by associativity of $-\otimes_ZZ'$ alone.

	Part (3): \'etale-locally, $\tau_P$ is built entirely inside $\mathrm{CycSp}$, so
	$TC^-$, $TP$, $TC$ transport, applied to $\tau_P$, as functors of the cyclotomic structure
	\cite{nik}. Since $A$ (hence $THH(A)$) is connective, the fixed points
	computing $TR^r$ are similarly recovered from this structure \cite{nik}. Thus, $TR^r$
	transports as well, \'etale-locally via $\tau_P$, and globally under the same  assumption as in
	Theorem \ref{thmazumaya}. 
\end{proof}

	\begin{table}[h]
\centering
\small
\label{tabsummary}
\begin{tabular}{@{}p{3.6cm}p{4.3cm}p{4.3cm}@{}}
\toprule
 & \textbf{Theorem \ref{thmazumaya}}($A$ Azumaya$/Z$) & \textbf{Theorem \ref{thmderived}} ($A$ connective derived Azumaya$/Z$) \\
\midrule
$\Phi_A:THH(A)\to THH(Z)$, underlying spectra & Global; $\pi_0\Phi_A=\mathrm{trd}$ & Global; $\pi_0\Phi_A=$ reduced Morita trace (Lemma \ref{lempi0azumaya}) \\
Local Morita trace $\tau_P$, as cyclotomic spectra & \'Etale-local (Lemma \ref{lemmorita-pic}\,(i)); not claimed equal to $\Phi_A$ beyond $\pi_0$; global if $[A]=0\in\mathrm{Br}(Z)$ (converse not claimed) & Same \\
$TR^r$ & \'Etale-local via $\tau_P$ / global as above & Same, using connectivity of $A$ \\
$TC^-,TP,TC$ & \'Etale-local via $\tau_P$ / global as above & Same \\
Transported filtration $F^\bullet THH(A;\mathbb{Z}_p)$ & Defined globally (pullback along underlying $\Phi_A$); graded pieces as stated & Not treated in this paper for the derived case \\
Canonicity of $\Phi_A$ & Not claimed beyond Lemma \ref{leminvtrivial} classification & Same \\
\bottomrule
\end{tabular}

\medskip 
\caption{Global vs.\ \'etale-local content of the two main theorems.}
\end{table}

\section{Conclusion and discussion}
\label{furtho}

In this paper, we have established a rigorous homotopical foundation for transporting the Bhatt-Morrow-Scholze (BMS) motivic filtration to the noncommutative setting. By demonstrating that topological Hochschild homology descends to the center (as an underlying spectrum, and \'etale-locally as a cyclotomic spectrum)  for Azumaya algebras over quasisyntomic rings (Theorem \ref{thmazumaya}) and for connective derived Azumaya algebras over connective $\mathbb{E}_\infty$-rings (Theorem \ref{thmderived}), we have resolved the structural prerequisites for a noncommutative BMS filtration within this broad class. In the non-derived case of Theorem \ref{thmazumaya}, the underlying equivalence is used to construct the transported filtration  (Definition \ref{deftransported}), which is complete, exhaustive, and properly graded. In the derived case of Theorem \ref{thmderived}, by contrast, what is supplied is only the underlying $THH$-comparison that transport would require, not a constructed derived filtration (Table \ref{tabsummary}). The higher invariants $TR$, $TC^-$, $TP$, and $TC$ are preserved \'etale-locally, and globally whenever a global splitting module exists.

We mention the physical motivation: an application and interpretation of
the mathematical results above. What is
proved is the Azumaya/center $THH$-comparison of Theorems \ref{thmazumaya} and
\ref{thmderived} and the transported filtration of Definition \ref{deftransported}, for
the specific algebras (modular Weyl algebras, quantum tori at flux prime to $p$)
computed in Section \ref{azumayo}. From a physical perspective, we consider these
computations as an arithmetic setup for the reduction of quantum observable algebras
modulo $p$. The mathematical separation of the transported, center-dependent filtration
from the quantum data (the latter mentioned instead by the Brauer class and
$K$-theoretic pairings)  suggests, as motivation, a picture of the
quantum-to-classical interface relevant to the integer and fractional quantum Hall
effects. We do not claim a theorem identifying our filtration with any specific
physical observable beyond the algebraic statements already proved.

Looking forward, the establishment of this filtration for Azumaya algebras clarifies the boundaries and future direction for a fully general noncommutative prismatic cohomology
$\mathcal{NC}\text{-}\mathrm{Prismatic}^{i,j}$.
 The frontier now shifts to the non-Azumaya regimes, where Theorem \ref{thmazumaya} simply does not apply (its Azumaya hypothesis is not met) and, on the evidence available, the pattern it establishes for Azumaya algebras does not extend. The sharpest available measure of that non-extension is the computation of Chan and Lindenstrauss \cite{cl} for a maximal order $V$ in a central simple $\mathbb{Q}$-algebra $A$. Writing $e_p$ for the degree of the central division algebra $D_p$ over $\mathbb{Q}_p$ with $A\otimes_{\mathbb{Q}}\mathbb{Q}_p\cong M_{k_p}(D_p)$, the groups $THH_*(V)$ consist of $THH_*(\mathbb{Z})$ together with an extra summand $\bigoplus_p\mathbb{F}_p^{\oplus e_p-1}$ in every even degree. Since $e_p=1$ exactly at the primes where $A$ is unramified, that is, exactly where $V\otimes\mathbb{Z}_p$ is Azumaya over $\mathbb{Z}_p$, the extra summand is supported on the ramified locus and quantifies the deviation from $THH$ of the center once Azumaya-ness is dropped. We present this as a piece of evidence for a specific maximal-order example, demonstrating the failure of the center-comparison phenomenon in that particular non-Azumaya setting,  instead of   as a general theorem that every non-Azumaya algebra acquires additional $THH$ summands.
 
 The integral interpolation of the Hofstadter regime over mixed characteristic discrete valuation rings, as highlighted in Remark \ref{remppower}, and the study of general $\hbar$-adic deformations prior to specialization remain important open problems. Resolving these will likely require moving beyond Azumaya descent and trace methods toward  noncommutative Hochschild-Witt complexes \cite{Kaledin2016} and non-abelian structural setup capable of capturing the full $A_\infty$-cyclotomic deformation.

Beyond the Azumaya class, the existing general answer to constructing a motivic filtration outside the quasisyntomic commutative setting is the even filtration of Hahn-Raksit-Wilson \cite{hrw}, defined for $\mathbb{E}_\infty$-ring spectra and agreeing with the Bhatt-Morrow-Scholze filtration of Theorem \ref{thmhm} in the quasisyntomic commutative case. It does not  resolve the noncommutative, non-Azumaya problem of Section \ref{ncbms}. The even filtration is built from the ring spectrum  $\mathbb{E}_\infty$-structure, whereas the postulates of Section \ref{physo} ask for a filtration attached to a noncommutative 
 deformation $\mathcal{A}_\star$ together with compatibility under Morita equivalence of the resulting associative algebra - a different, and in our view harder, request. Combining the two perspectives is a natural next step that we leave open.

The material of this paper can be also useful in mathematical physics
\cite{RSZ, Zu2, Zu3} and condensed matter theory 
\cite{chernodub2017scale, zhang2020influence, zubkov2012momentum, 
zubkov2017topology, kmmzz}. 

\section*{Acknowledgments}

The authors were supported by the Institute of Mathematics, Academy of Sciences of the Czech
Republic (RVO 67985840). 
The authors would like to thank Petr Somberg for sharing his paper \cite{HKS} with us.

	\medskip
	\noindent\textbf{Data Availability.}
	Data sharing is not applicable to this article as no datasets were generated
	or analysed during the current study.

	\medskip
	\noindent\textbf{Declarations}

This paper arose in close collaboration with ChatGPT, an AI assistant developed by 
OpenAI. The collaboration involved exploring proof strategies, refining arguments, locating
references, and drafting. Revision and writing of final text have been done by the authors.
The authors take full responsibility for the mathematical content and the final text.

	\medskip
	\noindent\textbf{Conflict of interest.}
	The authors have no conflicts of interest to declare that are relevant to the
	content of this article.


\end{document}